\documentclass[12pt]{article}
\usepackage{ct}
\usepackage[shortlabels]{enumitem}
\usepackage{multicol}
\usepackage{arydshln}

\newtheorem{notation}[theorem]{Notation}
\specs{n}{vol (issue)}{year}{}

\dateline{Feb 17, 2024}{Mar 24, 2025}{TBD}
\newcommand{\f}    [1]{{\mathbb{F}_{#1}}}
\newcommand{\trace}[1]{{\mathrm{Tr}       \left({#1}\right)}}
\newcommand{\tr}   [2]{{\mathrm{Tr}_{{#1}}\left({#2}\right)}}
\newcommand{\Gcd}  [2]{{\mathrm{gcd}({#1},{#2})}}
\newcommand{\cubes}[1]{{\mathcal{C}_{#1}}}

\newcommand{\DE}[2]{{\mathsf{D}_{#1}^{{#2}}}}
\newcommand{\ProjectiveLine}[0]{{\mathbb{P}^1}}
\newcommand{\Family}[1]{\mathcal{F}_{#1}}

\newcommand{\K}{\mathbb K}
\newcommand{\N}{\mathbb N}
\newcommand{\Z}{\mathbb Z}

\newcommand{\F}{\mathbb F}
\newcommand{\M}{\mathbb M}

\newcommand{\lms}{\{\!\!\{}
\newcommand{\rms}{\}\!\!\}}

\DeclareMathOperator{\Gal}{Gal}

\DeclareMathOperator{\GL}{GL}

\DeclareMathOperator{\diag}{diag}
\DeclareMathOperator{\Aut}{Aut}
\DeclareMathOperator{\ELM}{ELM}
\DeclareMathOperator{\CCZ}{CCZ}
\DeclareMathOperator{\EL}{EL}
\DeclareMathOperator{\EA}{EA}

\DeclareMathOperator{\im}{Im}
\DeclareMathOperator{\GammaL}{\Gamma L}

\title{Equivalences of biprojective almost perfect nonlinear functions}
\author[1]{Faruk G\"olo\u{g}lu\thanks{Supported by {\sf GA\v{C}R Grant 18-19087S - 301-13/201843}.}}
\author[2]{Lukas K\"olsch}

\affil[1]{%
Department of Algebra, Charles University, Prague, Czech Republic.

\email{farukgologlu@gmail.com}%
}

\affil[2]{%
Department of Mathematics, Unviersity of South Florida, St. Petersburg, FL, United States.

\email{lukas.koelsch.math@gmail.com}%
}
\begin{document}

\maketitle

\begin{abstract}
Two important problems on almost perfect nonlinear (APN) functions are
the enumeration and equivalence problems. In this paper, we solve these
two problems for any biprojective APN function family by introducing a 
group theoretic method for those functions. Roughly half of the known 
APN  families of functions on even dimensions are biprojective. By our 
method, we settle the equivalence problem for all known biprojective 
APN functions. Furthermore, we give a new family of such functions. 
Using our method, we count the number of inequivalent APN functions in 
all known biprojective APN families and show that the new family found 
in this paper gives exponentially many new inequivalent APN functions. 
Quite recently, the Taniguchi family of APN functions was shown to contain
an exponential number of inequivalent APN functions by Kaspers and Zhou 
(J. Cryptol. \textbf{34}(1), 2021) which improved their previous count
(J. Comb. Th. A \textbf{186}, 2022) for the Zhou-Pott family. 
Our group theoretic method substantially simplifies the work required for 
proving those results and provides a generic natural method for every 
family in the large super-class of biprojective APN functions that 
contains these two family along with many others. 
\end{abstract}

\maketitle


\section{Introduction}
Almost perfect nonlinear (APN) functions are cryptographically important 
functions since they give optimal protection against differential attacks 
when used as a building block of a Substitution Permutation Network. 
These functions are combinatorially interesting as there are several 
connections between APN functions and other combinatorial objects like 
difference sets, distance-regular graphs, symmetric association schemes, 
uniformly packed codes, and dimensional dual hyperovals  
(see for instance~\cite{vanDam,graphs,dsets,hyperovals2,hyperovals}).
Important questions in the study of APN functions are to find new infinite 
families of (esp. bijective) APN functions, determine equivalences between 
known functions, enumerating inequivalent APN functions in total or within 
known families. 

There are at least 15 known families of quadratic APN functions 
(see \cite{WQL} for the most up-to-date list) and 
roughly half of them fall into the framework of $(q,r)$-biprojective functions 
introduced in \cite{GIEEE}. In this paper, we will first give a new infinite 
family of $(q,r)$-biprojective APN functions (Theorem \ref{thm_apn}). 
Then we will develop a technique for determining equivalence of two APN 
functions if one of them is $(q,r)$-biprojective (Theorem \ref{thm:projequiv}). 
Using this technique we are able to count the number of inequivalent functions 
in all of the known $(q,r)$-biprojective APN function families 
(Theorem~\ref{thm:counts}). We also show that, apart from one case, all of the 
known $(q,r)$-biprojective APN function families are pairwise inequivalent 
(Theorem \ref{thm:between}). The standard way of showing equivalence is by using 
computers to compare invariants in small dimensions. Apart from that, some 
previous results that show inequivalence between some specific infinite families 
exist (see for instance \cite{WQL,kasperszhouZP,kasperszhou}), usually relying 
on long and technical calculations with linear polynomials tailored towards the 
specific families. In this paper, we will give for the first time a generic 
method for determining equivalence of APN functions for the highly fertile 
super-class of $(q,r)$-biprojective functions that contains roughly half of the 
known families (Theorem \ref{thm:projequiv}). The novelty relies on exploiting 
the existence of a large cyclic subgroup in the automorphism group of the APN 
function. A similar approach has been succesfully employed for the 
algebraically similar object of \emph{semifields}~\cite{GK,golouglu2023counting}. 

Moreover, we are going to show that the new family we present contains an exponential number 
of inequivalent APN functions (Corollary \ref{cor:f4}). 
Among the known infinite families, this is only the second 
family with this property. We also count the number of inequivalent APN functions in all
other $(q,r)$-biprojective families of APN functions, thus completely settling the equivalence 
question within and between all biprojective families (Table \ref{table_biproj}). 
Our group theoretic framework allows us to substantially  
simplify the problem, without which the treatment of the more complex functions would 
not be possible. Recently, Kaspers and Zhou \cite{kasperszhou} showed 
that the Taniguchi family of APN functions contains an exponential number of inequivalent 
functions using an intricate analysis of linearized polynomials. With a similar method, 
they computed the number of inequivalent APN functions in Zhou-Pott family \cite{kasperszhouZP}.
With our group theoretic method, which works for any $(q,r)$-biprojective family, 
we can, in particular, recover those results in a more natural way and
provide similar results for all known biprojective APN families of Table \ref{table_biproj} that are out of reach of the method of Kaspers and Zhou.
Some of our ideas are inspired from the works of Dempwolff and Yoshiara 
\cite{DempwolffPower,YoshiaraPower}, which also employ group theory, but 
only cover the equivalence question of monomials, which have a much simpler structure 
than the biprojective functions we consider here. 

{
In Section \ref{sec_pre}, we recall the basic definitions including projective
and biprojective polynomials. Then in Section \ref{sec_apn}, we will prove that a
new family of biprojective APN functions exists. In Section \ref{seq_equiv}, we 
introduce our technique. 
In Sections \ref{sec:inside} and \ref{sec_between}, we prove our 
equivalence and enumerations results. The main results here are Theorem~\ref{thm:counts} 
and \ref{thm:between}. 
A special case of Theorem~\ref{thm:counts} is treated In Appendix \ref{sec:carlet}
where we concentrate on the $(1,q)$-biprojective APN family
introduced by Carlet, and provide equivalence and enumeration results specific to
that family. Finally, in Appendix \ref{sec_walsh}, we compute the Walsh spectrum of the new family.
}

\section{Preliminaries} \label{sec_pre}

{An  \textbf{$(n,n)$-vectorial Boolean function} is a map from the $n$-dimensional 
$\f{2}$-vector space $\F_{2}^n$ to itself. In this
paper we are only interested in the case where $n=2m$ is even. In that case, we can identify $\F_{2}^n$ with $\F_{2^m}\times \F_{2^m}$, which is the setting that we will be using throughout the paper.} 
An $(n,n)$-vectorial Boolean function $F$ is said to be \textbf{almost perfect
nonlinear (APN)} if $F(x) + F(x + a) = b$ has zero or two solutions for all
nonzero $a \in \F_2^n$ and all $b \in \F_2^n$.  Note that these functions are optimal in characteristic two since solutions
to $F(x) + F(x + a) = b$ come in pairs, i.e., $F(x_0+a) + F(x_0) = F((x_0+a)+a) + F(x_0+a)$. Recall further that the absolute trace map on a finite field $\f{2^m}$ is defined as
$\trace{x} = \sum_{i = 0}^{m-1} x^{2^i}$.

\subsection{Projective and biprojective polynomials}
Projective and biprojective APN functions were introduced by the first author 
in~\cite{GIEEE}. We recall the definitions in this section, as well as fix some 
notation we will use throughout the paper.

Let $\M$ be the finite field with $2^m$ elements. Let $F$ be
\begin{equation*}
F(x,y) = \left(f(x,y),g(x,y)\right),
\end{equation*}
with $q = 2^k$, $r = 2^l$, $k,l \ge 0$, and 
\begin{align*}
f(x,y) &= a_0 x^{q+1} + b_0 x^q y + c_0 x y^q + d_0 y^{q+1},\\
g(x,y) &= a_1 x^{r+1} + b_1 x^r y + c_1 x y^r + d_1 y^{r+1},
\end{align*}
where $a_i,b_i,c_i,d_i \in \M$.
We will call $f(x,y)$, $g(x,y)$ {\bf $q$-biprojective polynomials} and $F(x,y)$ a {\bf $(q,r)$-biprojective polynomial pair}.
We are going to use the shorthand notation
\begin{align*}
f(x,y) 	&= a_0 x^{q+1} + b_0 x^q y + c_0 x y^q + d_0 y^{q+1}\\
		&= (a_0,b_0,c_0,d_0)_q,\\
g(x,y)	&= a_1 x^{r+1} + b_1 x^r y + c_1 x y^r + d_1 y^{r+1}\\
		&= (a_1,b_1,c_1,d_1)_r.
\end  {align*}

The univariate polynomial $f(x,1)$ is called a \textbf{$q$-projective polynomial}. 
The careful study of zeroes of projective polynomials was done by Bluher \cite{bluher}.
Let $\ProjectiveLine(\M) = \M \cup \{\infty\}$ denote the \textbf{projective line} over 
the finite field $\M$, i.e., the ratios $v/w$ where 
$(v,w) \in \M \times \M \setminus \{(0,0)\}$ where $v/0$ for all $v \in \M^\times$ is
defined to be the symbol $\infty$.

Define
\begin{align*}
 \DE{f}{\infty}(x,y) &= a_0 x^q + a_0 x + c_0 y^q + b_0 y,\\
 \DE{g}{\infty}(x,y) &= a_1 x^r + a_1 x + c_1 y^r + b_1 y,
\end{align*}
and for $u \in \ProjectiveLine(\M) \setminus \{\infty\}$,
\begin{align*}
 \DE{f}{u}(x,y) &= (a_0 u + b_0) x^q + (a_0 u^q + c_0) x 
	 			+ (c_0 u + d_0) y^q + (b_0 u^q + d_0) y,
\end{align*}
and similarly
\begin{align*}
 \DE{g}{u}(x,y) &= (a_1 u + b_1) x^r + (a_1 u^r + c_1) x 
	 			+ (c_1 u + d_1) y^r + (b_1 u^r + d_1) y.
\end{align*} 

We will view $F$ as a vectorial Boolean function $F : \M \times \M \to \M \times \M$.
We will not make any distinction between polynomials and functions and call
a $(q,r)$-biprojective polynomial pair $F(x,y)$ a $(q,r)$-biprojective function $F$.
The following straightforward lemma that simplifies checking for the APN condition was proved in \cite{GIEEE}. 

\begin{lemma}\label{lem_APN}
Let $F(x,y) = (f(x,y),g(x,y))$ be a $(q,r)$-biprojective polynomial pair.
Then $F$ is APN on $\M \times \M$ if and only if 
$
\DE{f}{u}(x,y) = 0 = \DE{g}{u}(x,y)
$
has exactly two solutions for each $u \in \ProjectiveLine(\M)$.
\end{lemma}

Table \ref{table_biproj} lists all known biprojective APN families. 
 We denote the families of Gold \cite{Gold}, Carlet \cite{Carlet}, Taniguchi \cite{Taniguchi}, Zhou-Pott \cite{ZP13} functions and the two families contained in \cite{GIEEE} by $\mathcal{G}$, $\mathcal{C}$, $\mathcal{T}$, $\mathcal{ZP}$, $\mathcal{F}_1$, $\mathcal{F}_2$, respectively. The new family we present in Section~\ref{sec_apn} is denoted by $\mathcal{F}_4$ (family $\mathcal{F}_3$ in~\cite{GIEEE} is another family of biprojective APN functions that only contains sporadic examples). 
We want to note that the first component of the Taniguchi functions is often also written (in our notation) as $(1,0,c,d)_q$. However, it is easy to verify that all these function with $c \neq 0$ are equivalent to the $c=1$ case, and the $c=0$ case is a Zhou-Pott function. The \textbf{Count} column refers to the number of ($\CCZ$-)inequivalent functions in each family. 
For the precise definitions of equivalence, we refer to Section~\ref{seq_equiv}.

\begin{table}[ht] 
\noindent\begin{center} 
{\scriptsize
\begin{tabular}{|c|c|c|c|c|} 
\hline 
\textbf{Family} & \textbf{Function} &  \textbf{Notes} & \textbf{Count} & \textbf{Proved in}\\
\hline

$\mathcal{G}$ & $X^{q+1}$ &  
\begin{tabular}{@{}c@{}} $q=2^k$, $\Gcd{k}{m}=1$ \end{tabular} & 
  &    \\ 
\cdashline{2-3}[.8pt/0.5pt]
                & $((0,1,1,0)_q,(1,0,1,1)_q)$ &  $m$ odd. &$\frac{\varphi(2m)}{2}$ &\cite{Gold} \\
\cdashline{2-3}[.8pt/0.5pt]
                & $((1,0,b,a)_q,(0,1,1,b+1)_q)$ &  $m$ even, $\tr{\M/\f{2}} a=1$, 
								$b = \sum_{i = 0}^{k-1} a^{2^i}$. & & \\
\hline 

$\mathcal{C}$   & $(xy,(1,b,c,d)_{q})$ & 
\begin{tabular}{@{}c@{}}  $q=2^k$, $0 < k < m$,\\ $\Gcd{k}{m}=1$,\\ $x^{q+1}+bx^q+cx+d \ne 0$ for $x \in \M$. \end{tabular} & \begin{tabular}{@{}c@{}}$\frac{\varphi(m)}{2} $\\ Thm.~\ref{thm:carlet_1},~\ref{thm:counts} \end{tabular}
& \cite{Carlet}\\ 
\hline 

$\mathcal{T}$   & $((1,0,1,d)_q,(0,0,1,0)_{q^2})$ & 
\begin{tabular}{@{}c@{}}  $q=2^k$, $0 < k < m$,\\ $\Gcd{k}{m}=1$,\\ $x^{q+1}+x+d \ne 0$ for $x \in \M$. \end{tabular} & \begin{tabular}{@{}c@{}}$\geq \frac{\varphi(m)}{2}\lceil \frac{2^m+1}{3m}\rceil $\\ \cite{kasperszhou} \end{tabular}
& \cite{Taniguchi}\\ 
\hline 

$\mathcal{ZP}$  & $((1,0,0,d)_q,(0,0,1,0)_r)$ & 
\begin{tabular}{@{}c@{}} $q=2^k,r=2^j$, $0 < j,k < m$, $m$ even \\ $\Gcd{k}{m}=1$,\\
$d \ne a^{q+1}(b^q+b)^{1-r}$ for $a,b \in \M$. \end{tabular} & \begin{tabular}{@{}c@{}}$\frac{\varphi(m)}{2}\lfloor \frac{m}{4}+1\rfloor$\\ \cite{kasperszhouZP} \end{tabular}
& \cite{ZP13}\\ 
\hline 

$\Family{1}$  & $((1,0,1,1)_q,(1,1,0,1)_{q^2})$ & 
\begin{tabular}{@{}c@{}} $q=2^k, \quad 0 < k < m$,\\ $\Gcd{3k}{m}=1.$ \end{tabular} & \begin{tabular}{@{}c@{}}$\frac{\varphi(m)}{2} $\\ Thm.~\ref{thm:counts}\end{tabular}
& \cite{GIEEE}\\ 
\hline 

$\Family{2}$  & $((1,0,1,1)_q,(0,1,1,0)_{q^3})$ & 
\begin{tabular}{@{}c@{}} $q=2^k, \quad 0 < k < m$,\\ $\Gcd{3k}{m}=1$, $m$ odd. \end{tabular} & \begin{tabular}{@{}c@{}}$\frac{\varphi(m)}{2} $\\ Thm.~\ref{thm:counts}\end{tabular}
& \cite{GIEEE}\\ 
\hline 

$\Family{4}$  & $((1,0,0,B)_q,(0,1,\frac{a}{B},0)_r)$                    & 
\begin{tabular}{@{}c@{}} $q=2^k,\quad r=2^{k+m/2}, \quad 0 < k < m$,\\ $m\equiv 2 \pmod 4$, \quad $\Gcd{k}{m}=1,$ \\ 
$a \in \K^\times,\quad B \in \M^\times \setminus \cubes{\M}, \quad B^{q+r} \ne a^{q+1}.$\end{tabular} & \begin{tabular}{@{}c@{}}$\geq \frac{\varphi(m)}{2m}(2^\frac{m}{2}-2) $\\ Cor.~\ref{cor:f4}\end{tabular}
 & Thm. \ref{thm_apn} \\ 
\hline 

\end{tabular} 
}
\end{center}
\caption{Known infinite families of biprojective APN functions on $\M \times \M$} \label{table_biproj}
\end{table}


\section{The new APN family} \label{sec_apn}
\subsection{Family $\Family{4}$}
We will now present a new family of $(q,r)$-biprojective APN functions on 
$\M \times \M$ where $q=2^k$ with $\Gcd{k}{m}=1$ and $r=2^{k+m/2}$. We will use the following notation in this section: 
\begin{notation} 
\begin{itemize}
\item Let $\F = \f{2^n}$ with $n=2m$ even.

\item Let $\M = \f{2^m}$ be the finite field with $2^m$ elements where $m$ is even and $m/2$ is odd.

\item Let $\K = \f{2^{m/2}}$ be the finite field with $Q = 2^{m/2}$ elements.

\item Let $q = 2^k$ and $r = 2^{k+m/2}$ with $1 \le k \le m-1$.

\item Let $\cubes{\M} = (\M^\times)^3$ be the set of non-zero cubes of $\M$.

\item Clearly, $(\M^\times)^{Q+1} = \K^\times \subset \cubes{\M}$ (note $3|Q+1$ since $m/2$ is odd).

\item Let the group of $(Q+1)^{\textrm{-st}}$ roots of unity be denoted by $(\M^\times)^{Q-1}$. It is easy to see
that $(\M^\times)^{Q-1} \cap \K = \{1\}$ and any $x \in \M^\times$ can be written uniquely as $x = cg$ where $c \in \K^\times$
and $g \in (\M^\times)^{Q-1}$.
\end{itemize}
\end{notation}

We start with some simple lemmas.
\begin{lemma} \label{lem:basics}
Let $\Gcd{k}{m}=1$ and $r=qQ$. Then
\begin{itemize}
\item $\Gcd{q+1}{2^m-1} = \Gcd{r-1}{2^m-1} = \Gcd{q^2-1}{2^m-1} = 3$,
\item $\Gcd{q-1}{2^m-1} = \Gcd{r+1}{2^m-1} = \Gcd{q+1}{Q-1} = 1$.
\end{itemize}
\end{lemma}
\begin{proof}
We only prove $\Gcd{r-1}{2^m-1}=3$ and $\Gcd{r+1}{2^m-1} = 1$, the other statements are obvious/well-known.

Since $\Gcd{k}{m}=1$ and $m$ is even, we know that $k$ is odd. Then $\Gcd{k+m/2}{m}=2$. Indeed, assume $d$ is an odd divisor of $m$ and $k+m/2$, then it is also a divisor of $m/2$ and thus also of $k$, so $d=1$. Then $\Gcd{r-1}{2^m-1} = 2^\Gcd{k+m/2}{m}-1=3$ and $\Gcd{r+1}{2^m-1} = 1$ since $m/\Gcd{k+m/2}{m}$ is odd. 
\end{proof}
We are now ready to prove the APN property of the new family.
\begin{theorem} \label{thm_apn}
Let $B \in \M^\times \setminus \cubes{\M}$ and $a \in \K^\times$ be such that $B^{q+r} \ne a^{q+1}$, and
let $F : \M \times \M \to \M \times \M$ be defined as
\[
F(x,y) = ((1,0,0,B)_q,(0,1,a/B,0)_r),
\]
where $q=2^k$ {with $m \equiv 2 \pmod 4$}, $\Gcd{k}{m}=1$ and $r=qQ$. 
Then $F$ is APN.
\end{theorem}
\begin{proof}
We are going to use Lemma \ref{lem_APN}. 
First
\begin{align*}
\DE{f}{0}(x,y) &= B(y^q + y) = 0, \quad \textrm{and},\\
\DE{g}{0}(x,y) &= x^r + \frac{ax}{B} = 0,
\end{align*}
implies $y \in \f{2}$ and $x = 0$ are the only common solutions since
$\Gcd{r-1}{2^m-1} = 3$ and $a/B \not\in \cubes{\M}$. Similarly,
\begin{align*}
\DE{f}{\infty}(x,y) &= y^q + y = 0, \quad \textrm{and},\\
\DE{g}{\infty}(x,y) &= x + \frac{ax^r}{B} = 0,
\end{align*}
have the same common solutions. Now for $u \in \M^\times$,
\begin{align*}
\DE{f}{u}(x,y) &= ux^q + u^qx + B(y^q + y) = 0, \quad \textrm{and},\\
\DE{g}{u}(x,y) &= x^r + \frac{a}{B}x + \frac{a}{B} u y^r + u^r y = 0.
\end{align*}
Now,
\begin{align*}
\DE{f}{u}(ux,y) &= u^{q+1}(x^q + x) + B(y^q + y) = 0, \quad \textrm{and},\\
\DE{g}{u}(ux,y) &= u^r (x^r + y) + \frac{a}{B} u (x + y^r) = 0.
\end{align*}
When $x,y \in \f{2}$, the only solutions are $(x,y) \in \{ (0,0),(1,1) \}$
since $u^r + au/B = 0$ implies $a/B$ is a cube. 
We will proceed to show that these are the only solutions for $x,y \in \M$. 
Now we can assume $x,y \in \M \setminus \f{2}$, since $x \in \f{2}$ implies
$y \in \f{2}$ and vice versa for $\DE{f}{u}(ux,y) = 0$. Furthermore, 
$x = y^r$ implies $y^{r^2} + y = 0$ {and thus $x = y \in \f{4} \setminus \f{2}$ by Lemma~\ref{lem:basics}. Then}
 $x^q + x = y^q + y = 1$, and in turn $u^{q+1} = B$, which is impossible since
$\Gcd{q+1}{2^m-1}=3$ and $B \not\in \cubes{\M}$. The same argument shows
$y \ne x^r$ and we can concentrate on
\begin{align}
u^{q+1} &= \frac{B(y^q + y)}{x^q + x}, \textrm{ and}, \label{eq1}\\
u^{r-1} &= \frac{a(x + y^r)}{B(x^r+y)},               \label{eq2}
\end{align}
for $x,y \in \M \setminus \f{2}$ with $x^r \ne y$ and $x \ne y^r$.
Now assume \eqref{eq1} and \eqref{eq2} holds for such $x,y \in \M \setminus \f{2}$, 
and let
\begin{align}
\phi_q(x,y) &= \frac{y^q + y}{x^q + x} = \frac{1}{cg}, \textrm{ and}, \label{eq3}\\
\phi_r(x,y) &= \frac{x + y^r}{x^r+y}   = dh,                          \label{eq4}
\end{align}
for some $c,d \in \K^\times$ and $g,h \in (\M^\times)^{Q-1}$. Multiplying \eqref{eq1} and
\eqref{eq2} we get
\[
\phi_q(x,y)\phi_r(x,y) = \frac{u^{q+r}}{a} \in \K, 
\]
and therefore $g = h$. Noting that $q^2\equiv q^2Q^2\equiv r^2 \pmod{2^m-1}$, we get
\[
1 = \frac{u^{(q+1)(q-1)}}{u^{(r-1)(r+1)}} 
  = \frac{B^{q+r}\phi_q(x,y)^{q-1}}{a^{q+1}\phi_r(x,y)^{r+1}},
\]
which implies 
\begin{equation}\label{eq5}
(cg)^{q-1}(dg)^{r+1} = c^{q-1}d^{q+1}g^{q+r} = c^{q-1}d^{q+1} = \frac{B^{q+r}}{a^{q+1}} \ne 1 
\end{equation}
since by our assumption $B^{q+r} \ne a^{q+1}$. Now let $z = y + x^Q$ and consider
\begin{align*}
 cg(z^q + z + (x^q + x)^Q) &= x^q + x, \textrm{ and},\\
 (x^q+x) + z^r             &= dg((x+x^q)^Q + z),
\end{align*}
which is a rewriting of \eqref{eq3} and \eqref{eq4}. Or,
\begin{align*}
 cg(x^q + x)^Q + (x^q + x) &= cg(z^q + z), \textrm{ and},\\
 dg(x^q + x)^Q + (x^q + x) &= z^r + dgz.
\end{align*}
Adding the two equations and also adding $d$ times the former and $c$ times the
latter equation, we get the following two equations:
\begin{align}
 (c+d)g(x^q + x)^Q &= z^r  + cgz^q  + (c+d)gz, \textrm{ and},\label{eq:5}\\
 (c+d) (x^q + x)   &= c(z^r + dgz^q).\label{eq:6}
\end{align}
Now if $c+d = 0$ then either $z = y + x^Q = 0$ and $c = d = 1$ since $(x^q+x)^{Q-1}\in (\M)^{Q-1}$
by \eqref{eq3} and 
\eqref{eq4}, which contradicts \eqref{eq5}; or $z^r = cgz^q$ by the
above equations and consequently $z^{r-q} = z^{q(Q-1)}= cg$, and again we reach the
contradiction $c = d = 1$ by \eqref{eq5}. If $c + d \ne 0$ then comparing $g$ times \eqref{eq:6} to the power $Q$ with \eqref{eq:5} yields
\begin{align*}
(c(z^r + dgz^q))^Q g &= z^r  + cgz^q  + (c+d)gz,\\
cgz^q + cdz^r &= z^r + cgz^q + (c+d)gz,\\
(cd+1)z^r &= (c+d)gz,
\end{align*}
that is to say $z^{r-1} = \frac{(c+d)g}{cd+1}$, thus $g \in \cubes{\M}$. 
Note that $cd = 1$ implies $c = d$ contrary to our assumption. But now by \eqref{eq1},
\[
u^{q+1} = \frac{B}{cg},
\]
a contradiction since the left hand side is a cube and the right hand side is not.
\end{proof}

\section{A technique to determine equivalence} \label{seq_equiv}

We will now develop a technique that allows us to determine when 
$(q,r)$-biprojective functions are equivalent or not. 
Denoting by $\F=\F_{2^n}$, we will first recall the different 
types of equivalences. Denote by $\Gamma_F = \{(x,F(x)) \colon x \in \F\}$ 
the \textbf{graph} of $F$.
\begin{definition}

	Two functions $F,G \colon \F \rightarrow \F$ are called
	\begin{enumerate}[label=(\roman*)]
		\item \emph{$\CCZ$-equivalent}, if there exists an affine permutation  
		\[\mathcal{A} : (x,y) \mapsto 
		\begin{pmatrix}
			M & K \\
			N & L
		\end{pmatrix}
		\begin{pmatrix}
			x \\
			y
		\end{pmatrix}
		+
		\begin{pmatrix}
			u \\
			v
		\end{pmatrix}
		\] 
		of $\F \times \F$ such that 
		$\mathcal{A}(\Gamma_F) = \Gamma_G$;

		\item \emph{extended affine equivalent (EA-equivalent)}, \\
			if $F$ and $G$ are $\CCZ$-equivalent with $K = 0$;
		\item \emph{extended linear equivalent (EL-equivalent)}, \\
			if $F$ and $G$ are $\EA$-equivalent with $(u,v) = (0,0)$;
		\item \emph{affine equivalent}, \\
			if $F$ and $G$ are $\EA$-equivalent with $K = N = 0$;
		\item \emph{linear equivalent}, \\
			if $F$ and $G$ are affine equivalent with $(u,v) = (0,0)$.
	\end{enumerate}
\end{definition}
EA-equivalence between $F,G$ can be written equivalently as $N(x)+L(F(x))=G(M(x))$ which is readily checked from the definition.

A major result by Yoshiara~\cite[Theorem 1]{YoshiaraQuadratic} states that two 
quadratic APN functions $F$ and $G$ are $\CCZ$-equivalent if and only if they are 
EA-equivalent. It is then straightforward that under the additional condition 
$F(0)=G(0)=0$ it even suffices to consider EL-equivalence 
(see for instance~\cite[Proposition 2.2.]{kasperszhou}). We summarize these
observations in the following theorem.

\begin{theorem} \label{thm:yoshiara}
Two quadratic APN functions $F,G \colon \F \rightarrow \F$ with $F(0)=G(0)=0$ are $\CCZ$-equivalent if and only if they are EL-equivalent.
\end{theorem}

We define the group of EL-mappings (i.e., the set of mappings that correspond to extended linear transformations on graphs) as 
\begin{equation*}
\ELM = \Bigg\{ \begin{pmatrix}
	M & 0 \\
	N & L
\end{pmatrix} \in \GL(\F \times \F)\Bigg\}.
\end{equation*}
Note that when we refer to linear mappings in this work, we always refer to linearity over the prime field $\F_2$, for instance $\GL(\F) \cong \GL(n,\F_2)$ and $\GL(\F \times \F) \cong \GL(2n,\F_2)$. Further denote by 
\[\Aut_{\EL}(F)=\{\mathcal{A} \in \ELM \colon \mathcal{A}(\Gamma_F)=\Gamma_F\}\]
 the group of EL-automorphisms of a function $F$. Clearly, if $F$ and $G$ are EL-equivalent, the corresponding EL-automorphism groups are conjugate in $\ELM$. We include the simple proof (essentially just a special case of the orbit-stabilizer theorem) for the convenience of the reader.
\begin{proposition} \label{prop:conjugated}
	Assume $F,G \colon \F \rightarrow \F$ are EL-equivalent via the EL-mapping $\gamma \in \ELM$, i.e., $\gamma(\Gamma_F)=\Gamma_G$. Then $\Aut_{\EL}(F)=\gamma^{-1} \Aut_{\EL}(G) \gamma$.
\end{proposition}
\begin{proof}
	Assume $\delta \in \Aut_{\EL}(G)$. Then $\gamma^{-1} \delta \gamma \in \Aut_{\EL}(F)$. Indeed
	\[(\gamma^{-1} \delta \gamma) (\Gamma_F)=(\gamma^{-1} \delta)(\Gamma_G) = \gamma^{-1}(\Gamma_G) = \Gamma_F.\]
	We conclude $\Aut_{\EL}(F) \subseteq \gamma^{-1} \Aut_{\EL}(G) \gamma$.
	The other inclusion follows by symmetry.
\end{proof}

\subsection{The group theoretic framework}

In this section, we will develop a framework that allows us to prove EL-inequivalence of two functions $F,G$ that are both biprojective polynomial pairs. 

Generally, it is very difficult to determine when two (APN) functions are equivalent or not. Usually, it is only possible to check inequivalence in low dimensions via computer using certain invariants. Thus, the usual way to argue that a family of APN functions is ``new'' is to show that they are not equivalent to the known APN functions in low dimensions. This is of course on several levels unsatisfying: One does not gain any theoretical insight as to \emph{why} the mappings are inequivalent, and one does not get any information on the behavior in larger dimensions. 

The key idea in our approach is inspired by the inequivalence results of Yoshiara~\cite{YoshiaraPower} and Dempwolff~\cite{DempwolffPower} on power functions. In both papers, the authors establish $\CCZ$-in\-equi\-valence between power functions by exploiting the existence of a large cyclic subgroup in the automorphism groups of power functions. Similarly, we will identify a large cyclic subgroup in the group $\Aut_{\EL}(F)$ if $F$ is a biprojective polynomial pair and use this subgroup and Proposition~\ref{prop:conjugated} to prove inequivalences between biprojective polynomial pairs. Note that the general approach of this technique is quite flexible to determine (in)equivalence of combinatorial objects; it was for instance also used to determine isotopies for semifields in~\cite{GK}.

We start by fixing some further notation that will be used throughout this section:
\begin{notation}
\begin{itemize}
	\item $\F = \f{2^n}$, $\M = \f{2^m}$ with $n = 2m$.
	\item $q=2^k$, $r=2^l$, $\overline{q}=2^{m-k}$, $\overline{r}=2^{m-l}$.
	\item We denote by $p$ a primitive divisor of $2^m-1$, i.e., a prime that divides $2^m-1$ but not $2^i-1$ for $i<m$. Such a prime exists if $m>1$, $m \neq 6$ by Zsygmondy's theorem (cf. \cite[Chapter IX., Theorem 8.3.]{HuppertII}). Note that $p \neq 2$ since $p|2^m-1$.
	\item $P$ is the unique Sylow $p$-subgroup of $\M^\times$. 
	\item For $a \in \M^\times$ we denote by $m_a\in \GL(\M)$ the linear map $x \mapsto ax$. 
	\item For $A,B,C,D \in \GL(\M)$ we write diagonal matrices as $\diag(A,B)=\begin{pmatrix}
		A & 0 \\
		0 & B
	\end{pmatrix} \in \GL(\F)$ and $\diag(A,B,C,D)=\begin{pmatrix}
		A & 0 &0&0\\
		0 & B &0 & 0 \\
		0& 0& C & 0 \\
		0& 0& 0& D
	\end{pmatrix} \in \GL(\F \times \F)$.
	\item $Z^{(q,r)} = \{ \diag(m_a,m_a,m_{a^{q+1}},m_{a^{r+1}})\colon a \in \M^\times\}$ is a cyclic subgroup of $\GL(\F \times \F)$ of order $2^m-1$.
\item $Z^{(q,r)}_P = \{ \diag(m_a,m_a,m_{a^{q+1}},m_{a^{r+1}})\colon a \in P\}$ is the unique Sylow $p$-subgroup of $Z^{(q,r)}$ of order $|P|$.
\item For a $(q,r)$-biprojective polynomial pair $F$ we denote by $C_F$ the centralizer of $Z_P^{(q,r)}$ in $\Aut_{\EL}(F)$ (see Lemma~\ref{lem:subgroups} for a proof that $Z_P^{(q,r)} \leq \Aut_{\EL}(F)$).
	\end{itemize}
\end{notation}

The reason to consider $Z^{(q,r)}$ is the simple (but crucial) fact that $Z^{(q,r)} \leq \Aut_{\EL}(F)$ for any $(q,r)$-biprojective polynomial pair $F$.
\begin{lemma} \label{lem:subgroups}
	Let $F=(f_1(x,y),f_2(x,y))$ be a $(q,r)$-biprojective polynomial pair. Then 
	\[Z^{(q,r)}_P \leq Z^{(q,r)} \leq \Aut_{\EL}(F).\]
\end{lemma}
\begin{proof}
The first inclusion is obvious. 
A simple calculation yields
\[\begin{pmatrix}
	m_a & & & \\
	 & m_a & & \\
	 & & m_{a^{q+1}} & \\
	& & & m_{a^{r+1}}
\end{pmatrix}  \begin{pmatrix}
	x \\
	y \\
	f_1(x,y) \\
	f_2(x,y)
\end{pmatrix} = \begin{pmatrix}
	ax \\
	ay \\
	a^{q+1}f_1(x,y) \\
	a^{r+1}f_2(x,y)
\end{pmatrix}. \]
We have $a^{q+1}f_1(x/a,y/a) = f_1(x,y)$ and $a^{r+1}f_2(x/a,y/a) = f_2(x,y)$, so
\[\Bigg\{\begin{pmatrix}
	ax \\
	ay \\
	a^{q+1}f_1(x,y) \\
	a^{r+1}f_2(x,y)
\end{pmatrix}\colon x,y \in \M \Bigg\} = \Bigg\{\begin{pmatrix}
	x \\
	y \\
	f_1(x,y) \\
	f_2(x,y)
\end{pmatrix}\colon x,y \in \M \Bigg\},\]
so $Z^{(q,r)} \leq \Aut_{\EL}(F)$ as claimed.
\end{proof}

{A key idea of the approach is the following: Instead of working with the entire automorphism group $\Aut_{\EL}(F)$ (which is in general very difficult to determine), we focus only on the much simpler subgroup $Z^{(q,r)}$. The following lemma makes this possible by establishing that under certain conditions $Z_P^{(q,r)}$ is not only a Sylow subgroup of $Z^{(q,r)}$, but also of $\Aut_{\EL}(F)$. }The proof of the Lemma is inspired by a similar result in~\cite[Corollary 2]{YoshiaraPower}.

\begin{lemma} \label{lem:zpsylow}
	Let $m>2$, $m \neq 6$ and $F \colon \F_{2^n} \rightarrow \F_{2^n}$ be a $(q,r)$-biprojective polynomial pair such that $C_F$ contains $Z^{(q,r)}$ as a subgroup with index $I$ and $p$ does not divide $I$. Then $Z_P^{(q,r)}$ is a Sylow $p$-subgroup of $\Aut_{\EL}(F)$.
\end{lemma}
\begin{proof}
	First define $S = \{\diag(m_a,m_b,m_c,m_d) \colon a,b,c,d \in P\}$. Clearly, $|S|=|P|^4$. We show that $S$ is a Sylow $p$-subgroup of $\GL(\F\times\F)$. We have $|\GL(\F\times\F)| = |\GL(4m,\F_2)| = 2^{2m(4m-2)} \prod_{i=1}^{4m}(2^i-1)$. Clearly, $2^{km+i}-1 \equiv 2^i-1 \pmod{2^m-1}$ for $k \in \N$. As $p$ is a $2$-primitive divisor of $2^m-1$, all integers $2^j-1$ in $[1,4m]$ are coprime to $p$, except for $j=m,2m,3m,4m$. The $p$-part of $2^{2m}-1=(2^m-1)(2^m+1)$ and $2^{4m}-1=(2^{2m}-1)(2^{2m}+1)$ is clearly $|P|$ (recall that $p \neq 2$). With a little more effort, $2^{3m}-1 = (2^m-1)(2^{2m}+2^m+1)$ has also $p$-part $|P|$ since $2^{2m}+2^m+1 = 3+(2^m-1)+(2^{2m}-1)$ is not divisible by $p$ because $p \neq 3=2^2-1$ if $m>2$.
	
	We conclude that the $p$-part of $|\GL(\F\times\F)|$ is $|P|^4$ and $S$ is a Sylow $p$-subgroup of $\GL(\F\times\F)$ as claimed.
	
	In particular, by the second Sylow theorem, all Sylow $p$-subgroups of $\GL(\F\times\F)$ are abelian since $S$ is abelian. Then also the Sylow $p$-subgroups of $\Aut_{\EL}(F)$ are abelian. Let $R$ be such an abelian Sylow $p$-subgroup of $\Aut_{\EL}(F)$ that contains the $p$-group $Z_P^{(q,r)}$. Since $R$ is abelian, $R$ is a subgroup of the centralizer $C_F$ of $Z_P^{(q,r)}$ in $\Aut_{\EL}(F)$, which contains $Z^{(q,r)}$ as an index $I$ subgroup. Since $p$ does not divide $I$ by our assumption, we know that $R = Z_P^{(q,r)}$ and $Z_P^{(q,r)}$ is a Sylow $p$-subgroup of $C_F$.
\end{proof}

The following lemma was proven in~\cite{GK}.
\begin{lemma} \label{lem:yoshiara}
Let 
\[ 
	M_P = \{\diag(m_a,m_a) \colon a \in P\},
\]
and 
\[
	M = \{\diag(m_a,m_a) \colon a \in \M^\times\}.
\]
Let $N_{\GL(\F)}(M_P), N_{\GL(\F)}(M)$ and $C_{\GL(\F}(M_P), C_{\GL(\F)}(M)$ be the 
normalizers and centralizers of $M$ and $M_P$ in $\GL(\F)$. Then
	\begin{enumerate}[(a)]
		\item \begin{align*}
            N_{\GL(\F)}(M_P) &= N_{\GL(\F)}(M) = \left\{\begin{pmatrix}
			m_{c_1} \tau & m_{c_2} \tau \\
			m_{c_3} \tau&m_{c_4} \tau
		\end{pmatrix} \colon \begin{array}{l} 
		                              c_1,c_2,c_3,c_4 \in \M^\times, \\ 
		                              \tau \in \Gal(\M/\F_2)
							 \end{array} \right\}
		\cap \GL(\F) \\&\cong \GammaL (2,\M),
        \end{align*}
		\item \begin{align*}
        C_{\GL(\F)}(M_P) &= C_{\GL(\F)}(M) = \left\{\begin{pmatrix}
			m_{c_1}  & m_{c_2}  \\
			m_{c_3}&m_{c_4} 
		\end{pmatrix} \colon c_1,c_2,c_3,c_4 \in \M^\times\right\}\cap \GL(\F)\\&\cong \GL (2,\M).
        \end{align*}
	\end{enumerate}
\end{lemma}	

\subsection{A theorem to determine equivalence of projective polynomial pairs}

First, we compute $\lambda^{-1} \diag(m_a,m_a,m_{a^{q+1}},m_{a^{r+1}}) \lambda$ for $\lambda \in \ELM$. We have for $\lambda = \begin{pmatrix}
	M & 0 \\
	N & L
\end{pmatrix}$ with $M,L \in \GL(\F)$ that $\lambda^{-1} =  \begin{pmatrix}
	M^{-1} & 0 \\
	L^{-1}NM^{-1} & L^{-1}
\end{pmatrix}$ and 
\begin{equation}
\lambda^{-1} \diag(m_a,m_a,m_{a^{q+1}},m_{a^{r+1}}) \lambda = \begin{pmatrix}
	M^{-1}\diag(m_a,m_a)M & 0 \\
 \ast & L^{-1}\diag(m_{a^{q+1}},m_{a^{r+1}})L
\end{pmatrix}.
\label{eq:conjugation}
\end{equation}

\begin{theorem} \label{thm:projequiv}
Assume $m>2$ and $m \neq 6$. Let $F$ and $G$ be $(q_1,r_1)$- and $(q_2,r_2)$-biprojective polynomial pairs respectively, with $q_i = 2^{k_i}$ and $r_i = 2^{l_i}$ where $k_1,l_1 \neq m/2$ and $k_1\not\equiv \pm l_1 \pmod m$. 
Assume further that
	\begin{equation}
		p \text{ does not divide } [C_F:Z^{(q_1,r_1)}].
	\tag{C}\label{eq:condition}
	\end{equation}
	Then $F,G$ are EL-equivalent if and only if they are EL-equivalent via a mapping $\gamma = \begin{pmatrix} M & 0 \\
	N & L \end{pmatrix}$ where $M,L \in \GL(\F)$ and, writing $M=\begin{pmatrix}
	 M_1 & M_2 \\
	M_3 & M_4
 \end{pmatrix},\;L=\begin{pmatrix}
	 L_1 & L_2 \\
	L_3 & L_4
 \end{pmatrix}$, the subfunctions $M_1,M_2,M_3,M_4,L_1,L_2,L_3,L_4 \colon \M \rightarrow \M$ are monomials (written as linearized polynomials) of the same degree, say $2^t$, or zero. 
	Furthermore, we have either 
	\begin{itemize}
		\item 	$L_2=L_3=0$ ,
		\item $k_1\equiv \pm k_2 \pmod m$, $l_1 \equiv \pm l_2\pmod m$,
		\item If $k_1\equiv k_2 \pmod m$ (resp. $l_1 \equiv l_2\pmod m$), then $L_1$ (resp. $L_4$) has degree $2^t$,
		\item If $k_1\equiv -k_2 \pmod m$ (resp. $l_1 \equiv -l_2\pmod m$), then $L_1$ (resp. $L_4$) has degree $2^{m-t}$,
	\end{itemize}
	or 
	\begin{itemize}
		\item 	$L_1=L_4=0$,
		\item $k_1\equiv \pm l_2 \pmod m$, $l_1 \equiv \pm k_2\pmod m$,
		\item If $l_1\equiv k_2 \pmod m$ (resp. $k_1 \equiv l_2\pmod m$), then $L_2$ (resp. $L_3$) has degree $2^t$,
		\item If $l_1\equiv -k_2 \pmod m$ (resp. $k_1 \equiv -l_2\pmod m$), then $L_2$ (resp. $L_3$) has degree $2^{m-t}$.
	\end{itemize}
	If additionally $k_1,l_1 \not\equiv 0 \pmod m$ then $N=0$, that is $F$ and $G$ are even linear equivalent.
\end{theorem}
\begin{proof}
Assume $F,G$ are EL-equivalent by the EL-mapping $\alpha\in \ELM$, i.e., $\alpha(\Gamma_F)=\Gamma_G$. For $\alpha^{-1} Z_P^{(q_2,r_2)} \alpha \leq \Aut_{\EL}(F)$ we have $|\alpha^{-1} Z_P^{(q_2,r_2)} \alpha| = |P| = |Z_P^{(q_1,r_1)}|$, so $\alpha^{-1} Z_P^{(q_2,r_2)} \alpha$ is a Sylow $p$-subgroup of  $\Aut_{\EL}(F)$ by Lemma~\ref{lem:zpsylow} as long as condition~\eqref{eq:condition} holds.  In particular, it is conjugate to $Z_P^{(q_1,r_1)}$ in $\Aut_{\EL}(F)$ by Sylow's theorem, i.e., there exists a $\lambda \in \Aut_{\EL}(F)$ such that 
\begin{equation}\label{eq:conjugated}
\lambda^{-1} \alpha^{-1} Z_P^{(q_2,r_2)} \alpha \lambda = (\alpha \lambda)^{-1} Z_P^{(q_2,r_2)} (\alpha \lambda) = Z_P^{(q_1,r_1)}.
\end{equation}
Note that $F,G$ are also EL-equivalent by the mapping $(\alpha \lambda)\in \ELM$ since $(\alpha\lambda)(\Gamma_F)=\alpha(\Gamma_F)=\Gamma_G$.
 Writing 
\[(\alpha \lambda) =  \begin{pmatrix}
	M & 0 \\
	N & L
\end{pmatrix}\] with $M,L \in \GL(\F)$, this means $G \circ M = L \circ F +N$. We immediately get using Eqs.~\eqref{eq:conjugated} and~\eqref{eq:conjugation} that $\diag(m_a,m_a)M=M\diag(m_b,m_{b})$ for all $a \in P$ and some $b=\pi(a)$, where $\pi \colon P \rightarrow P$ is a bijection, i.e., $M$ is in the normalizer of $\{\diag(m_a,m_a) \colon a \in P\}$. From Lemma~\ref{lem:yoshiara} we deduce that $M_1,M_2,M_3,M_4$ are zero or monomials of the same degree, say $2^t$.

We now compute $G \circ M = L \circ F +N $. 

Writing $F,G$ as biprojective polynomial pairs, we have $F(x,y) = (f_1(x,y),f_2(x,y))$, $G(x,y) = (g_1(x,y),g_2(x,y))$ with 
\[g_1(x,y)=\alpha_1x^{2^r+1}+\beta_1x^{2^r}y+\gamma_1xy^{2^r}+\delta_1y^{2^r+1}\text{, }g_2(x,y)=\alpha_2x^{2^s+1}+\beta_2x^{2^s}y+\gamma_2xy^{2^s}+\delta_2y^{2^s+1}\]
and
\[(G \circ  M) (x,y) = (g_1(M_1(x)+M_2(y),M_3(x)+M_4(y)),g_2(M_1(x)+M_2(y),M_3(x)+M_4(y))).\]
Denoting by $g_1'$ and $g_2'$ the two components of $G\circ M$, we get 
\begin{align*}
	g_1'(x,y) =& \alpha_1((c_1x^{2^t}+c_2y^{2^t})^{q_2+1})+\beta_1(c_1x^{2^t}+c_2y^{2^t})^{q_2}(c_3x^{2^t}+c_4y^{2^t})\\
	&+\gamma_1(c_1x^{2^t}+c_2y^{2^t})(c_3x^{2^t}+c_4y^{2^t})^{q_2}+\delta_1(c_3x^{2^t}+c_4y^{2^t})^{q_2+1} \\
	g_2'(x,y) =& \alpha_2((c_1x^{2^t}+c_2y^{2^t})^{r_2+1})+\beta_2(c_1x^{2^t}+c_2y^{2^t})^{r_2}(c_3x^{2^t}+c_4y^{2^t})\\
	&+\gamma_2(c_1x^{2^t}+c_2y^{2^t})(c_3x^{2^t}+c_4y^{2^t})^{r_2}+\delta_2(c_3x^{2^t}+c_4y^{2^t})^{r_2+1}.
\end{align*}
We have $L\circ F = [L_1(f_1(x,y))+L_2(f_2(x,y)),L_3(f_1(x,y))+L_4(f_2(x,y))]$. 
Observe that the only terms that appear in $g_1'$ are of the form $x^{q_22^{t}+2^t}$, $x^{q_22^{t}}y^{2^t}$, $x^{2^{t}}y^{q_22^{t}}$ and $y^{q_22^{t}+2^{t}}$, and the same for $g_2'$ with $q_2$ substituted by $r_2$. Just by comparing these degrees to the possible degrees of $L \circ F$, we conclude that, since $k_1 \not\equiv \pm l_1 \pmod m$, either $L_2=L_3=0$ and 
\begin{equation}
	k_1\equiv \pm k_2 \pmod m, l_1\equiv \pm l_2 \pmod m
\label{eq:pm_1}
\end{equation}

or $L_1=L_4=0$ and 
\begin{equation}
k_1\equiv \pm l_2 \pmod m, k_2\equiv \pm l_1 \pmod m.
\label{eq:pm_2}
\end{equation}
 Furthermore, the non-zero $L_i$ for $i \in \{1,2,3,4\}$ are monomials of degree $2^t$ or $2^{m-t}$ depending on the signs in Eqs.~\eqref{eq:pm_1} and \eqref{eq:pm_2} since $k_1,l_1 \not\equiv m/2 \pmod m$.

We also deduce that if $k_1,l_1\not\equiv 0 \pmod m$, then also $k_2,l_2\not\equiv 0 \pmod m$ by our previous considerations, and both $G\circ M$ and $L \circ F$ have no linear terms, so that $N$ is the only linear part in the entire equation $G \circ M =  L \circ F +N$. This implies $N=0$.
\end{proof}

Theorem~\ref{thm:projequiv} essentially answers when two biprojective polynomial pairs are EL-equivalent as long as Condition~\eqref{eq:condition} on the centralizer stated in the theorem is satisfied. We will prove this condition later for most of the known APN biprojective polynomial pairs, which results in a comprehensive inequivalence result of the known APN functions that are constructed via projective polynomials. Note that the condition $m>2$ is not very restrictive as all APN functions on $\F_2^2$ and $\F_2^4$ are easy to classify by an exhaustive computer search.

\begin{remark}
	It is possible to derive a slightly more general version of Theorem~\ref{thm:projequiv} where the conditions $k_1,l_1 \neq m/2$ and $k_1\not\equiv \pm l_1 \pmod m$ are not needed, however the statement gets more involved. We decided to exclude these cases for the sake of readability since, in this paper, we will only need the cases that are covered by Theorem~\ref{thm:projequiv} as it stands.
\end{remark}

\section{Inequivalences within biprojective families} \label{sec:inside}

We now use Theorems~\ref{thm:yoshiara} and~\ref{thm:projequiv} to obtain 
$\CCZ$-inequivalence results inside the APN families that can be written 
as biprojective polynomial pairs. To do this, we need to check the 
technical condition~\eqref{eq:condition}, the proof of which is provided 
in Appendix \ref{sec_centralizer}. 

In the case of Zhou-Pott and Taniguchi functions, inequivalence check 
between these families was already done in~\cite{kasperszhou,kasperszhouZP}. 
Note that the approach in these papers is very technical, for instance 
the proof of the inequivalence of the Taniguchi 
functions in~\cite{kasperszhou} relies on an intricate investigation of certain 
linear polynomials that takes more than two dozen pages. 
Theorem~\ref{thm:projequiv} allows us to minimize the effort since we only have
to consider linear equivalences of a very special type. This also allows us to 
deal with the more complicated functions in the other families, which seem to be 
out of reach using the approach in \cite{kasperszhou,kasperszhouZP}. We skip an 
alternative proof of the Zhou-Pott and Taniguchi functions (which can easily be 
constructed based on the approach we outline for the other families), and 
instead only deal with the remaining families for which the (in)equivalence 
question has not been settled yet. In particular, note that 
Theorem~\ref{thm:counts} together with Theorem~\ref{thm:carlet_1} completely 
settles the equivalence problem for the family of Carlet functions, thus solving 
the problem left open in~\cite{kasperszhouZP}.

\begin{theorem} \label{thm:counts}
Let $q=2^k$ and $\overline{q} = 2^{m-k}$. We consider the following functions defined on $\M \times \M$ with $m>2, m \neq 6$. Recall the definitions of the functions in Table~\ref{table_biproj}.
\begin{enumerate}[(a)]
	\item Let $F_q,F_{q'}$ be two functions in the families $\mathcal{F}_1$ with parameters $q$ and $q'$
	\begin{enumerate}[(i)]
		\item $F_{q}$ is $\CCZ$-equivalent to $F_{q'}$ if and only if $q'=q$ or $q'=\overline{q}$.
		\item There are $\varphi(m)/2$ $\CCZ$-inequivalent functions in both families $\mathcal{F}_1,$ and $\mathcal{F}_2$.
	\end{enumerate}

		\item Let $F_q,F_{q'}$ be two functions in the families $\mathcal{F}_2$ with parameters $q$ and $q'$.
	\begin{enumerate}[(i)]
		\item $F_{q}$ is $\CCZ$-equivalent to $F_{q'}$ if and only if $q'=q$ or $q'=\overline{q}$.
		\item There are $\varphi(m)/2$ $\CCZ$-inequivalent functions in both families $\mathcal{F}_1,$ and $\mathcal{F}_2$.
	\end{enumerate}
	
	\item Let $F_{q,B,a}$ be a function in the family $\mathcal{F}_4$ with parameters $q,B,a$. 
		\begin{enumerate}[(i)]
		\item  $F_{q,B,a}$ is linear equivalent to $F_{\overline{q},B,a}$.
		\item For each choice of $B',B$ and $a'$ there exists an $a$ such that $F_{q,B,a}$ is linear equivalent to $F_{q,B',a'}$.
		\item $F_{q,B,a}$ is linear equivalent to $F_{q,B,a'}$ for at most $m$ choices of $a'$.
		\item There are in total at least $\varphi(m)(2^{m/2}-2)/(2m)$ $\CCZ$-inequivalent functions in Family $\mathcal{F}_4$.
	\end{enumerate}

	\item Let $C(q;1,\beta,\gamma,\delta):=(xy,(1,\beta,\gamma,\delta)_q)$ be a Carlet function with $m$ odd. 
	\begin{enumerate}[(i)]
		\item $C(q;1,\beta,\gamma,\delta)$ is $\CCZ$-equivalent to  $C(q';1,\beta_2,\gamma_2,\delta_2)$ if and only if $q'=q$ or $q'=\overline{q}$ (no matter the choice of the other coefficients).
		\item 	 There are in total $\varphi(m)/2$ $\CCZ$-inequivalent Carlet functions.
	\end{enumerate}

		\end{enumerate}
		Here $\varphi$ denotes Euler's totient function.
\end{theorem}

\begin{proof}
	By Theorem~\ref{thm:projequiv}, we just have to test for linear equivalence, i.e., for two functions $F,F'$ we have to check $F \circ M = L \circ F'$, where the subfunctions $L_i,M_i$ are zero or monomials, except for the Carlet family, where we do have to test for EL-equivalence, i.e., $F \circ M = L \circ F'+N$.
	
	For functions $F_q$, $F_{q'}$ in the families $\mathcal{F}_1,\mathcal{F}_2$, the inequivalences follow directly from  Theorem~\ref{thm:projequiv} except in the case where $q'=\overline{q}$. In this case, the equivalence can be stated directly: For both families, the equivalence $F_q \circ M = L \circ F_{\overline{q}}$ can be achieved by setting $M_1=M_4=L_2=L_3=0$, $M_2=M_3=x^{\overline{q}}$, $L_1=x$, and $L_4=x^{\overline{q}}$ for the Family $\mathcal{F}_1$ and $L_4=x^{q}$ for Family $\mathcal{F}_2$.

\underline{Family $\mathcal{F}_4$:} Consider $F_{q,B,a}$, $F_{q',B',a'}$. The two functions are $\CCZ$-inequivalent if $k \not\equiv \pm k' \pmod m$ by Theorem~\ref{thm:projequiv} and the fact that $k' \not\equiv \pm(k+m/2) \pmod m$ since (by the proof of Lemma~\ref{lem:basics}) we have $\Gcd{k+m/2}{m}=2$ . If $q'=\overline{q}$, we have $F_{q,B,a} \circ M = L \circ F_{\overline{q},B,B^{Q+1}/a}$ by setting $M_1=M_4=x^{\overline{q}}$, $L_1=x$, $L_4=(B^{Q}/a)x^{Q}$ and $M_2=M_3=L_2=L_3=0$. Observe that $(B^{Q+1}/a) \in \F_{Q}$, so the conditions of the family are not violated. 

It thus only remains to consider the case $q'=q$. Consider now $M_i=c_i^{2^t}x^{2^t}$. By Theorem~\ref{thm:projequiv}, this immediately gives $L_1=d_1x^{2^{t}}$, $L_4=d_4x^{2^{t}}$ and $L_2=L_3=0$. $F_{q,B,a}$ and $F_{q,B',a'}$ are then linear equivalent if and only if the following equations have a solution:
\begin{align*}
&(c_1x+c_2y)^{q2^{t}+2^t}+B(c_3x+c_4y)^{q2^{t}+2^t} = d_1(x^{q2^{t}+2^t}+B'^{2^t}y^{q2^{t}+2^t}), \\
	&(c_1x+c_2y)^{qQ2^{{t}}}(c_3x+c_4y) +a/B(c_1x+c_2y)(c_3x+c_4y)^{qQ2^{t}} \\&= d_4(x^{qQ2^{t}}y+(a'/B')^{2^t}xy^{qQ2^{t}}).
\end{align*}
Just like in the proof of Theorem~\ref{thm:projequiv}, we again compare the coefficients which leads to the equations

 \noindent\begin{minipage}{0.45\textwidth}
\begin{align*}
	c_1^{q+1}+Bc_3^{q+1} &= d_1 \\
	c_1^{q}c_2+Bc_3^{q}c_4 &= 0\\	
	c_1c_2^{q}+Bc_3c_4^{q} &= 0\\	
	c_2^{q+1}+Bc_4^{q+1} &= d_1B'^{2^t}.
\end{align*}
    \end{minipage}%
    \begin{minipage}{0.1\textwidth}\centering
    and 
    \end{minipage}%
    \begin{minipage}{0.45\textwidth}
	\begin{align*}
		c_1^{qQ}c_3+(a/B)c_1c_3^{qQ} &= 0 \\
	c_1^{qQ}c_4+(a/B)c_2c_3^{qQ} &= d_4\\	
	c_2^{qQ}c_3+(a/B)c_1c_4^{qQ} &= (a'/B')^{2^t}d_4\\	
	c_2^{qQ}c_4+(a/B)c_2c_4^{qQ} &= 0.
	\end{align*}
    \end{minipage}

From the second and third equation on the left, like in the corresponding case of Theorem~\ref{thm:projequiv}, we again infer that $c_2=c_3=0$ or $c_1=c_4=0$. 

\underline{Case $c_2=c_3=0$:} The first and last equations on the left imply $d_1=c_1^{q+1}$ and $B/(B'^{2^t})c_4^{q+1}=c_1^{q+1}$. In particular, $B/(B'^{2^t})=r^{q+1}$ must be a cube and $c_1/c_4=r \omega$ for $\omega \in \F_4^\times$. Note that, no matter the choice of $B,B'$ we can always find a $t$ such that $B/(B'^{2^t})$ is a cube. Indeed, $1/B' \in B\mathcal{C}_\M$ or $1/B' \in B^2\mathcal{C}_\M$. In the second case, $B/B' \in \mathcal{C}_\M$, so we can choose $t=0$, in the second case $B/(B'^2) \in \mathcal{C}_\M$, so we can choose $t=1$.

The second and third equations of the right can be satisfied if and only if
\[\left(\frac{c_1}{c_4}\right)^{qQ-1}=\frac{a}{a'^{2^t}}\frac{B'^{2^t}}{B} = \frac{a}{a'^{2^t}} \left(\frac{c_4}{c_1}\right)^{q+1},\]
or, equivalently, $(c_1/c_4)^{q(Q+1)} = a/(a'^{2^t})$. Note that $(c_1/c_4)^{q(Q+1)} \in \F_{Q}$, in particular there is for a fixed $a$ always precisely one $a'$ that satisfies the equation. We have thus shown that for arbitrary $a,B,B'$ there exists an $a'$ such that $F_{q,u',a'}$ is linear equivalent to $F_{q,u,a}$. We can thus assume from now on that $B'=B$.

Then $r^{q+1} = 1/B^{2^t-1}$, so $t$ is necessarily even. Then $(c_1/c_4)^{q(Q+1)} = r^{q(Q+1)}=a/(a'^{2^t})$, so for a fixed choice of $a,t$, precisely one $a'$ solves the equation. 

\underline{Case $c_1=c_4=0$:} From the two equations on the left, we deduce, setting $B=B'$: $(c_2/c_3)^{q+1} = B^{2^t+1}$, which implies that $t$ is odd. Similarly, from the two equations on the right, we infer $(c_3/c_2)^{qQ-1} = B^{2^t+1}/(a'^{2^t}a)$ and, combining this with the previous condition, $(c_3/c_2)^{q(Q+1)} = 1/(a'^{2^t}a)$. 
Just like in the last case, a fixed choice for $t,a$ yields a unique $a'$ that satisfies the condition. 

We conclude that, in total, for a function $F_{q,B,a}$ there are at most $m$ choices of $a'$ such that $F_{q,B,a'}$ are linear equivalent (since there are $m$ choices for $t$, with $t$ even and odd appearing in the first and second case, respectively). We have thus in total $\varphi(m)/2$ choices for (inequivalent) $q$, each of which yields at least $(2^{m/2}-2)/m$ inequivalent functions.
	
\underline{The Carlet family:} $C(q;1,\beta,\gamma,\delta)$ is linear equivalent to $C(\overline{q};1,\beta,\gamma,\delta)$ by setting $L_1=x$, $L_4=M_1=M_4=x^{\overline{q}}$, $M_2=M_3=0$. Two Carlet functions $C(q;1,\beta_1,\gamma_1,\delta_1)$ and $C(q_2;1,\beta_2,\gamma_2,\delta_2)$ are then EL-equivalent if and only if $q=q_2$ and there are coefficients $c_1,c_2,c_3,c_4,d_1,d_4$ such that (compare Eq.~\eqref{eq:cond1_carlet}, Eq.~\eqref{eq:cond2_carlet})
\begin{align}
	&(c_1x+c_2y)(c_3x+c_4y)= d_1xy+N_1(x)+N_2(y),\label{eq:cond2_carlet_equiv}\\
	&(c_1x+c_2y)^{q+1}+\beta_1(c_1x+c_2y)^{q}(c_3x+c_4y)+\gamma_1(c_1x+c_2y)(c_3x+c_4y)^{q}+\delta_1(c_3x+c_4y)^{q+1} \nonumber \\
	&= d_4( x^{q+1}+\beta_2 x^{q}y+\gamma_2 xy^{q}+\delta_2 y^{q+1})\label{eq:cond1_carlet_equiv} 
\end{align}
	Identically to the corresponding case of the proof of Lemma~\ref{lem:centralizers}, Eq.~\eqref{eq:cond2_carlet_equiv} is always satisfied by the unique choice $N_1(x)=c_1c_3x^2$, $N_2(y)=c_2c_4y^2$, $d_1=c_1c_4+c_2c_3=\det(M)\neq 0$. Eq.~\eqref{eq:cond1_carlet_equiv} is satisfied if and only if $(1,\beta_2,\gamma_2,\delta_2)_{q}$ is in the orbit of $(1,\beta_1,\gamma_1,\delta_1)_{q}$ under the action of $G$ defined in Section~\ref{sec:carlet}. By Lemma~\ref{lem:transitiv}, this action is transitive on the set of all biprojective polynomials such that $f(x,1)$ has no zeros in $\M$, so all such functions are in the same orbit. We conclude that two Carlet functions $C(q;\alpha_1,\beta_1,\gamma_1,\delta_1)$ and $C(q;\alpha_2,\beta_2,\gamma_2,\delta_2)$ are always EL-equivalent, no matter the choice of coefficients.
\end{proof}

We want to emphasize that this result implies that the new Family $\mathcal{F}_4$ we found in this paper is only the second known infinite family that yields exponentially (in $n$) many inequivalent APN functions, next to the Taniguchi family. Theorem~\ref{thm:counts} actually not only gives a lower bound, but also an upper bound on the number of inequivalent functions inside $\mathcal{F}_4$, in total yielding a quite tight estimate:
\begin{corollary}\label{cor:f4}
	Let $N(m)$ be the number of inequivalent functions inside Family $\mathcal{F}_4$ on $\M \times \M$ where $m \equiv 2 \pmod 4$. Then
	\[\frac{\varphi(m)(2^{m/2}-2)}{2m} \leq N(m) \leq \frac{\varphi(m)(2^{m/2}-2)}{2}.\]
\end{corollary}

\section{Inequivalences between biprojective families}\label{sec_between}

We now prove inequivalences between the different biprojective APN families. 
This, again, is only made possible by the simplifications that 
Theorem~\ref{thm:projequiv} allows. In 
Theorem~\ref{thm:carlet_1} it is shown that the Carlet family for $m$ even is contained 
in the Zhou-Pott family, so we need not consider that case.

\begin{theorem}\label{thm:between}
Let $m>2$, $m \neq 6$ and $F,G$ be $(q_1,r_1)$- and $(q_2,r_2)$-biprojective 
polynomial pairs, respectively, defined on $\M \times \M$ from distinct families from 
the following list (see Table \ref{table_biproj}):
\begin{multicols}{2}
	\begin{itemize}
		\item Gold functions $\mathcal{G}$,
		\item Zhou-Pott functions $\mathcal{ZP}$,
        \item Taniguchi functions $\mathcal{T}$, \\with $(q,m) \neq (2,4)$,
		\item Carlet functions $\mathcal{C}$ for $m$ odd,
		\item $\mathcal{F}_1$, with $(q,m) \neq (2,4)$,
		\item $\mathcal{F}_2$, 
		\item $\mathcal{F}_4$.
			\end{itemize}
    \end{multicols}
	Then $F,G$ are $\CCZ$-inequivalent.
\end{theorem}
\begin{proof}
The inequivalence between Gold, Zhou-Pott and Taniguchi functions was already shown in \cite{kasperszhou,kasperszhouZP} (although it is not particularly challenging to give an alternative proof with the help of Theorem~\ref{thm:projequiv}).

The other results follow from Lemma~\ref{lem:centralizers} and 
Theorems~\ref{thm:yoshiara} and~\ref{thm:projequiv}. 
The functions $F$ and $G$ can only be $\CCZ$-equivalent if 
$q_1 \in \{q_2,\overline{q_2}\}$ and 
$r_1 \in \{r_2,\overline{r_2}\}$ or 
$q_1 \in \{r_2,\overline{r_2}\}$ and  
$r_1 \in \{q_2,\overline{q_2}\}$. This already proves most 
inequivalences. The remaining cases have to be checked by 
hand for linear equivalence under the conditions given in 
Theorem~\ref{thm:projequiv}. It thus only remains to check 
the inequivalences between the families $\mathcal{F}_1$ 
and Taniguchi functions; and between Zhou-Pott functions 
and functions in the families $\mathcal{F}_1$, $\mathcal{F}_4$.  

So let $F=(f_1(x,y),f_2(x,y))$ be a Taniguchi or Zhou-Pott function. 
In both cases, we have $f_2 = xy^{r_1}$. Assume that $F$ is linearly 
equivalent to a function 
$G=(g_1(x,y),g_2(x,y))=((\alpha_1,\beta_1,\gamma_1,\delta_1)_{q_2},
                        (\alpha_2,\beta_2,\gamma_2,\delta_2)_{r_2})$. 
Then $G =L\circ F\circ M $ for some bijective mappings $M$ and $L$. 
Here, by Theorem~\ref{thm:projequiv}, the subfunctions $M_1,\dots,M_4$ 
are monomials of the same degree $2^t$ or zero, and either 
$L_1=L_4=0$ or $L_2=L_3=0$. 
Set $F \circ M = (h_1(x,y),h_2(x,y))$. Then 
\[
h_2(x,y) = (ax+by)^{2^{t}}(cx+dy)^{2^{t}r_1} 
         = (ac^{r_1}x^{r_1+1}+c^{r_1}bx^{r_1}y+ad^{r_1}xy^{r_1}+bd^{r_1}y^{r_1+1})^{2^t}
\]
 for some $a,b,c,d \in \M$. Since either $L_2=L_3=0$ or $L_1=L_4=0$ we then have either $L_4(h_2(x,y))=g_2(x,y)$ (in the first case) or $L_2(h_2(x,y))=g_1(x,y)$ (in the second case). If $L_2$ or $L_4$, respectively, is a monomial $ex^{2^{t'-t}}$, we can write 
\[L_2(h_2(x,y))=e(a'c'^{r_1}x^{r_12^{t'}+2^{t'}}+c'^{r_1}b'x^{r_12^{t'}}y^{2^{t'}}+a'd'^{2^{t'}}x^{2^{t'}}y^{r_12^{t'}}+b'd'^{r_1}y^{r_12^{t'}+2^{t'}})\]
 (or equivalently for $L_4$).  We now compare the coefficients of $x^{r_12^{t'}+2^{t'}}$, $x^{r_12^{t'}}y^{2^{t'}}$, $x^{2^{t'}}y^{r_12^{t'}}$, $y^{r_12^{t'}+2^{t'}}$. We get either $t'=0$ and the necessary conditions
\begin{equation*}
	ea'c'^{r_1} = \alpha_z,\;\;
	eb'c'^{r_1} = \beta_z,\;\;
	ea'd'^{r_1} = \gamma_z,\;\;
	eb'd'^{r_1} = \delta_z.
\end{equation*}
or $2^{t'}=\overline{r_1}$ and
\begin{equation*}
	ea'c'^{{r_1}} = \alpha_z,\;\;
	eb'c'^{{r_1}} = \gamma_z,\;\;
	ea'd'^{{r_1}} = \beta_z,\;\;
	eb'd'^{{r_1}} = \delta_z.
\end{equation*}
for $z=1$ or $z=2$. It is easy to see that these conditions cannot be satisfied if $\alpha_z=\beta_z=\delta_z=1$ and $\gamma_z=0$ or $\alpha_z=\gamma_z=\delta_z=1$ and $\beta_z=0$, which shows inequivalence to the family $\mathcal{F}_1$. Similarly, the conditions cannot be satisfied for $\alpha_z,\delta_z \neq 0$ and $\beta_z=\gamma_z=0$ or $\alpha_z=\delta_z=0$, $\beta_z,\gamma_z\neq 0$, which shows inequivalence to the family $\mathcal{F}_4$.
\end{proof}
Note that we exclude the cases $(q,m)=(2,4)$ for Family $\mathcal{F}_1$ and the Taniguchi family since the conditions of Theorem~\ref{thm:projequiv} are not satisfied.

\section*{Acknowledgements}
{We would like to thank the anonymous reviewers for their comments which improved the presentation of the results.}

\bibliographystyle{alphaurl}

\bibliography{apn_inequiv}

\appendix

\section{The centralizer condition}\label{sec_centralizer}
We now determine the centralizers to check condition \eqref{eq:condition} in Theorem~\ref{thm:projequiv}. This is not very difficult, but quite technical, since we will compute the centralizers for many different families of functions.

\begin{lemma} \label{lem:centralizers}
	Let $m>2$, $m \neq 6$ and let $F$ be a $(q=2^k,r=2^l)$-biprojective polynomial pair with $l \notin \{0,m/2\}$, $k \not\equiv \pm l \pmod m$ and $C_F$ the centralizer of $Z_P^{(q,r)}$ in $\Aut_{\EL}(F)$. Then
	\begin{itemize}
		\item If $F$ is a Zhou-Pott function with $j\neq 0$ or contained in the family $\mathcal{F}_4$, then $C_F = \bigcup_{\omega\in \F_4^\times}Z_\omega ^{(q,r)} $, where we define
		\[Z_\omega ^{(q,r)} = \{\diag(m_a,m_{a\omega},m_{a^{q+1}},m_{a^{r+1}\omega^{r}})\colon a \in\M^\times\}\]
		for $\omega \in \F_4^\times$. Observe that $Z_1 ^{(q,r)} = Z ^{(q,r)}$.
		\item If $F$ is a Taniguchi function, then $C_F = Z^{(q,r)}$.
		\item If $F$ is a function in the families $\mathcal{F}_1$ or $\mathcal{F}_2$, then $C_F = Z^{(q,r)}\cup A \cup B$, where
        \begin{align*}
		        A &= \bigg\{\begin{pmatrix}
			m_a & m_a & 0 & 0 \\
			m_a & 0 & 0 & 0 \\
			0 & 0 & m_{a^{q+1}} & 0 \\
			0 & 0 & 0 & m_{a^{r+1}}
		\end{pmatrix} \colon a \in \M^\times \bigg\}, \\ B &= \bigg\{\begin{pmatrix}
			0 & m_a & 0 & 0 \\
			m_a & m_a & 0 & 0 \\
			0 & 0 & m_{a^{q+1}} & 0 \\
			0 & 0 & 0 & m_{a^{r+1}}
		\end{pmatrix}\colon a \in \M^\times \bigg\}. 
        \end{align*}

		\item If $F$ is a Carlet function for odd $m$ then $C_F$ contains $Z ^{(q,r)}$ as an index $3$ subgroup.
	\end{itemize}
	In all these cases, $C_F=Z^{(q,r)}$ or $Z^{(q,r)}$ is an index $3$ subgroup of $C_F$. In particular, $p$ does not divide the index and condition \eqref{eq:condition} is satisfied for all families.

\end{lemma}
\begin{proof}
Assume $\lambda  = \begin{pmatrix}
	M & 0 \\
	N & L
\end{pmatrix}\in C_F$. Then by Eq.~\eqref{eq:conjugation}, $M$ is in the centralizer of $$\{\diag(m_a,m_a)\colon a \in P\}$$ in $\GL(\F)$. 
Thus, by Lemma~\ref{lem:yoshiara}, all possible mappings that are contained in $C_F$ satisfy
$M =\begin{pmatrix}
	m_{c_1} & m_{c_2} \\
	m_{c_3} & m_{c_4}
\end{pmatrix}$, i.e., we can write $M = \begin{pmatrix}
	M_1 & M_2 \\
	M_3 & M_4
\end{pmatrix}$ where the mappings $M_1,\dots,M_4$ are monomials $M_i= c_ix$. Since $\lambda  \in \Aut_{\EL}(F)$, we can proceed identically to the proof of Theorem~\ref{thm:projequiv} with $t=0$, and we infer $N=0$ and $L_1=d_1x$, $L_2=L_3=0$, $L_4=d_4x$, except when $k\equiv 0 \pmod m$, in which case $N$ may be nonzero. We check which of these mappings are contained in $\Aut_{\EL}(F)$. 

Let $F=[(\alpha_1,\beta_1,\gamma_1,\delta_1)_{q},(\alpha_2,\beta_2,\gamma_2,\delta_2)_{r}]$ be a biprojective polynomial pair. Spelling out the equation $F \circ M = L \circ F+N$ yields  
\begin{align}
	&\alpha_1((c_1x+c_2y)^{q+1})+\beta_1(c_1x+c_2y)^{q}(c_3x+c_4y)+\nonumber \\ &\gamma_1(c_1x+c_2y)(c_3x+c_4y)^{q}+\delta_1(c_3x+c_4y)^{q+1} 
	= d_1(\alpha_1x^{q+1}+\beta_1x^{q}y+\gamma_1xy^{q}+\delta_1y^{q+1})\label{eq:cond1} \\
	&\alpha_2((c_1x+c_2y)^{r+1})+\beta_2(c_1x+c_2y)^{r}(c_3x+c_4y)+\nonumber\\ &\gamma_2(c_1x+c_2y)(c_3x+c_4y)^{r}+\delta_2(c_3x+c_4y)^{r+1}
	= d_4(\alpha_2x^{r+1}+\beta_2x^{r}y+\gamma_2xy^{r}+\delta_2y^{r+1}),\label{eq:cond2}
\end{align}
with added linear function $N_1(x)+N_2(y)$ in Eq.~\eqref{eq:cond1} if $k\equiv 0 \pmod m$. 
We now consider Eqs.~\eqref{eq:cond1} and~\eqref{eq:cond2} for various cases of $F$. 

\textbf{Zhou-Pott functions}: Let $F$ be a Zhou-Pott function. Eqs.~\eqref{eq:cond1} and~\eqref{eq:cond2} yield
\begin{align}
(c_1x+c_2y)^{q+1}+d(c_3x+c_4y)^{q+1} &= d_1(x^{q+1}+dy^{q+1}) \label{eq:condZP1}\\
	(c_1x+c_2y)(c_3x+c_4y)^{r}&= d_4xy^{r}.\label{eq:condZP2}
\end{align}
Comparing the coefficients of these polynomials leads to several defining equations. We start with Eq.~\eqref{eq:condZP2}
\begin{equation*}
	c_1c_3^{r} = 0,  \;\;
	c_2c_3^{r} = 0,  \;\;
		c_1c_4^{r} = d_4, \;\;
	c_2c_4^{r} = 0.
\end{equation*}
 $d_4$ is nonzero, so these equations imply $c_2=c_3=0$ since $c_1$ and $c_4$ are necessarily nonzero. Substituting this in Eq.~\eqref{eq:condZP1} and comparing the coefficients yields
\begin{equation*}
	c_1^{q+1}= d_1, \;\;
	dc_4^{q+1} = d\cdot d_1. 
\end{equation*}
In total, we thus have $c_4^{q+1}=c_1^{q+1}=d_1$ and $c_1c_4^{r}=d_4$. Since $\Gcd{k}{m}=1$ and $m$ is even, the mapping $x \mapsto x^{q+1}$ is $3$-to-$1$ on $\M^\times$, so we conclude that $c_4 = \omega c_1$ with $\omega \in \F_4^\times$. Further $c_1c_4^{r}=c_1^{r+1}\omega^{r}=d_4$. 

\textbf{Taniguchi functions}: Outside of the special case $(k,m)=(1,4)$ (where we have $2k=m/2$), Eqs.~\eqref{eq:cond1} and~\eqref{eq:cond2} yield
\begin{align}
(c_1x+c_2y)^{q+1}+(c_1x+c_2y)(c_3x+c_4y)^{q}+d(c_3x+c_4y)^{q+1} &= d_1(x^{q+1}+xy^{q}+dy^{q+1}) \label{eq:condT1}\\
	(c_1x+c_2y)^{q^2}(c_3x+c_4y)&= d_4x^{q^2}y\label{eq:condT2}.
\end{align}
Comparing the coefficients again yields for Eq.~\eqref{eq:condT2}:
\begin{equation*}
	c_1^{q^2}c_3 = 0,  \;\;
	c_1^{q^2}c_4 = d_4,  \;\;	
	c_2^{q^2}c_3 = 0,  \;\;
	c_2^{q^2}c_4 = 0.
\end{equation*}
This again leads to $c_2=c_3=0$. Substituting this in Eq.~\eqref{eq:condT1} and comparing the coefficients now gives
\begin{equation*}
	c_1^{q+1}= d_1 ,  \;\;
	dc_4^{q+1} = d\cdot d_1,  \;\;
	c_1c_4^{q} =d_1.
\end{equation*}
Comparing the first and the third equation directly yields $c_1=c_4$, and $d_1=c_1^{q+1}$, $d_4=c_1^{q^2+1}$ immediately follow. 

\textbf{Family $\mathcal{F}_1$:}
Eqs.~\eqref{eq:cond1} and~\eqref{eq:cond2} yield similarly to the previous cases the conditions 

 \noindent\begin{minipage}{0.45\textwidth}
\begin{align}
	c_1^{q+1}+c_1c_3^{q}+c_3^{q+1}&= d_1, \label{eq:l1}\\
	c_1^{q}c_2+c_3^{q}c_2+c_3^{q}c_4 &= 0,\label{eq:l2}\\
	c_1c_2^{q}+c_1c_4^{q}+c_3c_4^{q} &=d_1,\label{eq:l3}\\
	c_2^{q+1}+c_2c_4^{q}+c_4^{q+1} &= d_1,\label{eq:l4}
	\end{align}
    \end{minipage}%
    \begin{minipage}{0.1\textwidth}\centering
    and 
    \end{minipage}%
    \begin{minipage}{0.45\textwidth}
	\begin{align}
		c_1^{q^2+1}+c_1^{q^2}c_3+c_3^{q^2+1} &=d_4,\label{eq:r1}\\
	c_1^{q^2}c_2+c_1^{q^2}c_4+c_3^{q^2}c_4 &=d_4,\label{eq:r2}\\
c_1c_2^{q^2}+c_2^{q^2}c_3+c_3c_4^{q^2} &=0,\label{eq:r3}\\
	 c_2^{q^2+1}+c_2^{q^2}c_4+c_4^{q^2+1} &=d_4.\label{eq:r4}
	\end{align}
    \end{minipage}\vskip1em
\underline{Case $c_1=c_3$:} \\
Note that because of Eq.~\eqref{eq:l1}, we have $d_1 =c_1^{q+1}$ and in particular $c_1=c_3\neq 0$. Eq.~\eqref{eq:l2} then shows $c_4=0$.  Eq.~\eqref{eq:r1} gives $d_4 = c_1^{q^2+1}$ and then Eq.~\eqref{eq:r2} yields $c_2=c_1$. It is then easy to verify that all other equations are satisfied, so this case yields the set $A$ defined in the theorem. \\
\underline{Case $c_1c_3=0$:} \\
Similar calculations to the one in the previous case lead to the conclusion that the only valid solutions are the ones given by the sets $B$ (for $c_1=0$) and $Z^{(q,r)}$ (for $c_3=0$) as defined in the statement of the theorem.\\
\underline{Case $c_1,c_3\neq 0$, $c_1 \neq c_3$:}\\
From Eq.~\eqref{eq:l2} and Eq.~\eqref{eq:r3}, we get
\begin{equation*}
c_2^{q^2} = \frac{c_3^{q^3}c_4^{q^2}}{(c_1+c_3)^{q^3}} = \frac{c_3c_4^{q^2}}{(c_1+c_3)},
\end{equation*}
which directly gives $c_3^{q^3}(c_1+c_3) = c_3(c_1+c_3)^{q^3}$ and (simplifying further) $c_3^{q^3}c_1=c_1^{q^3}c_3$ and finally $c_3^{q^3-1} = c_1^{q^3-1}$. Since $\Gcd{3k}{m}=1$, we have $\Gcd{q^3-1}{2^m-1}=1$, so the mapping $x \mapsto x^{q^3-1}$ is a bijection on $\M$. We conclude $c_1=c_3$, which contradicts our assumption. 

\textbf{Family $\mathcal{F}_2$}: The calculations are very similar to the $\mathcal{F}_1$ case and can be carried out on the same way with the same end result.

\textbf{The family $\mathcal{F}_4$}: 

Eqs.~\eqref{eq:cond1} and~\eqref{eq:cond2} yield the conditions:\\
 \noindent\begin{minipage}{0.45\textwidth}
\begin{align*}
	c_1^{q+1}+Bc_3^{q+1} &= d_1 \\
	c_1^{q}c_2+Bc_3^{q}c_4 &= 0\\	
	c_1c_2^{q}+Bc_3c_4^{q} &= 0\\	
	c_2^{q+1}+Bc_4^{q+1} &= d_1B.
\end{align*}
    \end{minipage}%
    \begin{minipage}{0.1\textwidth}\centering
    and 
    \end{minipage}%
    \begin{minipage}{0.45\textwidth}
	\begin{align*}
		c_1^{qQ}c_3+a/Bc_1c_3^{qQ} &= 0 \\
	c_1^{qQ}c_4+a/Bc_2c_3^{qQ} &= d_4\\	
	c_2^{qQ}c_3+a/Bc_1c_4^{qQ} &= a/Bd_4\\	
	c_2^{qQ}c_4+a/Bc_2c_4^{qQ} &= 0.
	\end{align*}
    \end{minipage}\vskip1em
Let us consider the case $c_1c_2c_3c_4 \neq 0$ first, and set $c_1=z_3c_1$ and $c_2=z_4c_2$ with $z_1,z_2 \in \M^\times$. The second and third equations on the left then yield $B = z_1^qz_2 = z_1z_2^q$, which implies $z_1^{q-1}=z_2^{q-1}$ and since $\Gcd{q-1}{2^m-1}=1$, we conclude $z_1=z_2$. Then $B=z_1^{q+1}$, in particular $B$ is a cube since $\Gcd{q+1}{2^m-1}=3$. So we arrive at a contradiction and $c_1c_2c_3c_4=0$. By the second and third equation on the left and the bijectivity of $M$, we infer that there are only two cases to check, $c_1=c_4=0$ and $c_2=c_3=0$.

\underline{Case $c_2=c_3=0$:} 
The second and third equation on the right yield $d_4=c_1^{qQ}c_4=c_1c_4^{qQ}$, so $c_1^{qQ-1} = c_4^{qQ-1}$ and $c_4=\omega c_1$ with $\omega \in \F_{4}^\times$ since $\Gcd{qQ-1}{2^m-1}=3$ by Lemma~\ref{lem:basics}. Similarly, the first and fourth equations on the left yield $d_1=c_1^{q+1}=c_4^{q+1}$ (observe that $\omega^{q+1}=1$). It is then easy to verify that for this choice of coefficients, all equations hold. So this case gives precisely the sets $Z_\omega ^{(q,r)}$ for $\omega \in \F_4^\times$. \\
\underline{Case $c_1=c_4=0$:} 
The first and fourth equations on the left give $d_1=Bc_3^{q+1}=c_2^{q+1}/B$, so $c_2^{q+1} = B^2c_3^{q+1}$. In particular, since $\Gcd{q+1}{2^m-1}=3$, we know that both $c_3^{q+1}$ and $c_2^{q+1}$ are cubes, but $B^2$ is not a cube since $B$ is a non-cube. So this case cannot occur. 

\textbf{The Carlet family:} Let $F =(xy,(1,\beta,\gamma,\delta)_{r})$ be a Carlet function. 
Eqs.~\eqref{eq:cond1} and~\eqref{eq:cond2} then yield
\begin{align}
	&(c_1x+c_2y)(c_3x+c_4y)= d_4xy+N_1(x)+N_2(y),\label{eq:cond2_carlet}\\
	&(c_1x+c_2y)^{r+1}+\beta(c_1x+c_2y)^{r}(c_3x+c_4y)+\gamma(c_1x+c_2y)(c_3x+c_4y)^{r}+\delta(c_3x+c_4y)^{r+1} \nonumber \\
	&= d_1( x^{r+1}+\beta x^ry+\gamma xy^r+\delta y^{r+1})\label{eq:cond1_carlet} 
\end{align}
	Observe that Eq.~\eqref{eq:cond2_carlet} can always be satisfied for given $c_1,c_2,c_3,c_4$ by the (unique) choice $d_4 = c_1c_4+c_2c_3=\det(M)\neq 0$, $N_1(x) = c_1c_3x^2$, $N_2(y)=c_2c_4y^2$. We thus only have to care about Eq.~\eqref{eq:cond1_carlet}. This equation is satisfied if and only if $(d_1,M)$ is in the stabilizer of $f(x,y)=(1,\beta,\gamma,\delta)_{r}$ under the action of $G$ (see Section~\ref{sec:carlet}). The size of this stabilizer is precisely $3(2^m-1)$ by Lemma~\ref{lem:transitiv}.
\end{proof}

\section{The Carlet family $\mathcal{C}$} \label{sec:carlet}
The case of $(1,r)$-biprojective polynomial pairs  requires some extra considerations since in this case the question of equivalence cannot be reduced to linear equivalence (see Theorem~\ref{thm:projequiv}) . The APN functions of this type are precisely the ones in the Carlet family $\mathcal{C}$ (see Table~\ref{table_biproj}). We will consider the special properties of this family in this section.

The following lemma can be easily deduced from a classical result by Bluher~\cite{bluher}, as is done for example in~\cite[Lemma 4]{bbsemifieldchar2}. 
\begin{lemma}[{\cite[Lemma 4]{bbsemifieldchar2}}]\label{lem:noroots}
	Let $q=2^k$ with $\Gcd{k}{m}=1$. The number of polynomials $f(x,y) = (p_1,p_2,p_3,p_4)_{q}$ with $p_1,p_2,p_3,p_4 \in \M$ such that $p_1\neq 0$ and $f(x,1)$ has no roots in $\M$ is precisely $\frac{(2^m+1)2^m(2^m-1)^2}{3}$. 
\end{lemma}
Note that for $f(x,y) = (p_1,p_2,p_3,p_4)_{q}$ with $p_1 \neq 0$ we have that $f(x,1)$ has no roots if and only if $f(x,y)=0$ holds only for $x=y=0$. Indeed, we have $f(x,y)=y^{q+1}f(x/y,1)$ for $y \neq 0$ and $f(x,0)=p_1x^{q+1}=0$ if and only if $x=0$.

Similarly to~\cite{bbsemifieldchar2}, we define a group action by $G =  \M^\times \times \GL(2,\M)$ on the set of biprojective polynomial pairs $f=(p_1,p_2,p_3,p_4)_q$ with $p_1 \neq 0$  as follows: $M=\begin{pmatrix}
	c_1 & c_3 \\
	c_2 & c_4
\end{pmatrix} \in\GL(2,\M)$ acts on a biprojective polynomial pair $f(x,y)=(p_1,p_2,p_3,p_4)_{q}$ via

\[Mf =\begin{pmatrix}
	x & y 
\end{pmatrix} M \begin{pmatrix}
	p_1 & p_3 \\
	p_2 & p_4
\end{pmatrix} (M^{q})^t \begin{pmatrix}
	x^{q} \\
	y^{q}
\end{pmatrix}
\]
where $M^t$ denotes as usual the transpose of $M$ and $M^{q}$ the matrix where every entry is taken to the ${q}$-th power. A straightforward calculation shows $Mf(x,y) = f(c_1x+c_2y,c_3x+c_4y)$. $\M^\times$ acts on $f$ by regular multiplication, i.e., for $a \in \M^\times$ we have $af =  (ap_1,ap_2,ap_3,ap_4)_{q}$. Note that $|G|=(2^m-1)\cdot |\GL(2,\M)|=(2^m-1)^3(2^m+1)2^m$. Further, if $f(x,1)$ has no roots in $\M$ then $((a,M)f)(x,1)$ also has no roots in $\M$. This is obvious in the case of the scalar multiplication, and for the matrix action we have $Mf(x,1)=f(c_1x+c_2,c_3x+c_4)=0$ if and only if $c_2=c_1x$ and $c_4=c_3x$, which leads to $\det(M)=c_1c_4+c_2c_3=xc_1c_3+xc_1c_3=0$. Moreover, $(a,M)f(x,y)=(p_1',p_2',p_3',p_4')_q$ for some values $p_1',p_2',p_3',p_4'$ with $p_1' \neq 0$. This is again obvious for the scalar multiplication, and for the matrix action, the coefficient of the $x^{q+1}$-term is (by simply expanding $f(c_1x+c_2y,c_3x+c_4y)$): 
\[p_1c_1^{q+1}+p_2c_1^qc_3+p_3c_1c_3^q+p_4c_3^{q+1}=f(c_1,c_3) \neq 0\]
since $(c_1,c_3)\neq (0,0)$. Thus we can view the action of $G$ purely on the set of projective polynomials $f=(p_1,p_2,p_3,p_4)_q$ with $p_1 \neq 0$ and $f(x,1)\neq 0$ for all $x \in \M$. 

 We have the following central result about this action:
\begin{lemma}\label{lem:transitiv}
	Let $q=2^k$ be fixed with $\Gcd{k}{m}=1$. Then $G$ acts transitively on the set of polynomials $f=(p_1,p_2,p_3,p_4)_{q}$ with $p_1 \neq 0$ such that $f(x,1)$ has no roots in $\M$, i.e., all such polynomials are in the same orbit under $G$. The size of the stabilizer of any polynomial in this set is $3(2^m-1)$.
\end{lemma}
\begin{proof}
	Consider first the case that $m$ is even. Set $f(x,y)=x^{q+1}+uy^{q+1}=(1,0,0,u)_q$ for a non-cube $u \in \M^\times$, and observe that $f(x,1)$ has no roots in $\M$.
	By~\cite[Lemma 6]{bbsemifieldchar2}, the stabilizer of $f$ under $G$ has order $3(2^m-1)$. By the orbit-stabilizer theorem, the size of the orbit is then 
	\[|G|/(3\cdot (2^m-1)) = 2^m(2^m-1)^2(2^m+1)/3,\]
	so the entire set of polynomials with no roots by Lemma~\ref{lem:noroots}.
	
	Now let $m$ be odd and consider $f(x,y)=(1,0,1,u)_{q}$ with $u\in \M^\times$ such that $f(x,1)$ has no roots in $\M$ (such an $f$ exists for any $q,m$ by results of Bluher~\cite{bluher}). An element $(a,M) \in G$ acts on $f$ by definition as follows
	\[(a,M)f = \begin{pmatrix}
	x & y 
\end{pmatrix} aM \begin{pmatrix}
	1 & 1 \\
	0 & u
\end{pmatrix} (M^{q})^t \begin{pmatrix}
	x^q \\
	y^q
\end{pmatrix}\]
Now assume that $(a,M)$ is in the stabilizer of $f$ and set $M=\begin{pmatrix}
	c_1 & c_3 \\
	c_2 & c_4
\end{pmatrix}$. Then 

\[aM \begin{pmatrix}
	1 & 1 \\
	0 & u
\end{pmatrix} (M^{q})^t=\begin{pmatrix}
	1 & 1 \\
	0 & u
\end{pmatrix}.\] This implies that $a^2\det(M)^{q+1}=1$. Since $m$ is odd, the mapping $x \mapsto x^{q+1}$ permutes $\M$ and we can find a $\gamma \in \M$ such that $\gamma^{q+1}=a$. It is enough to consider the case $a=\det(M)=1$, all other cases for $a$ are covered by scaling $M$ by $\gamma$. So let $a=\det(M)=1$. In this case $((M^{q})^t)^{-1}=\begin{pmatrix}
	c_4^q &c_2^q \\
	c_3^q & c_1^q
\end{pmatrix}$. The condition 
\[ \begin{pmatrix}
	c_1 & c_3 \\
	c_2 & c_4
\end{pmatrix}\begin{pmatrix}
	1 & 1 \\
	0 & u
\end{pmatrix} =\begin{pmatrix}
	1 & 1 \\
	0 & u
\end{pmatrix}\begin{pmatrix}
	c_4^{q} &c_2^q \\
	c_3^q & c_1^q
\end{pmatrix}\] yields one equation per entry in the matrix:
\begin{align}
	c_1 &= c_4^q+c_3^q\label{eq:c1}  \\
	c_1+uc_3 &= c_2^q+c_1^q \nonumber\\
	c_2  &= uc_3^q \label{eq:c2} \\
	c_2+uc_4 &= uc_1^q\nonumber.
\end{align}
 Observe that all possible solutions of these equations yield an invertible matrix $M$ as long as $(c_3,c_4)\neq (0,0)$: Indeed, we have using Eqs.~\eqref{eq:c1} and~\eqref{eq:c2} $\det(M)=c_1c_4+c_2c_3 = c_4^{q+1}+c_3^qc_4+uc_3^{q+1}$ and $\det(M)\neq 0 $ no matter the choice of $(c_3,c_4)\neq (0,0)$ since $\det(M)\neq 0$ for $c_3 =0$, $c_4 \neq 0$ and if $c_3 \neq 0$, we substitute $c_4 \mapsto c_3c_4$ which gives $c_3^{q+1}c_4^{q+1}+c_3^{q+1}c_4+uc_3^{q+1}=c_3^{q+1}f(c_4,1)$ which is never $0$ since $f(x,1)$ has no root in $\M$.
After eliminating $c_1,c_2$ using Eqs.~\eqref{eq:c1} and~\eqref{eq:c2}, we get (after slightly reordering the equations)
\begin{align}
	c_4^q+ c_4^{q^2} = u^qc_3^{q^2} +c_3^{q^2}+c_3^q+uc_3 \label{eq:carletequiv_oben} \\
	uc_4 +uc_4^{q^2}= uc_3^q+uc_3^{q^2}. \label{eq:carletequiv}
\end{align}
	Dividing the second equation by $u$ and then adding the two equations yields
	\[c_4+c_4^q = uc_3+u^qc_3^{q^2}.\]
	Taking this equation to the power $q$ gives
	\[c_4^q+c_4^{q^2} = u^qc_3^q+u^{q^2}c_3^{q^3}.\]
	A comparison with Eq.~\eqref{eq:carletequiv_oben} yields
	\[u^qc_3^{q^2} +c_3^{q^2}+c_3^q+uc_3= u^qc_3^q+u^{q^2}c_3^{q^3}\]
	and, after slightly reordering, $u^{q}c_3^{q^2}+c_3^{q}+uc_3 = (u^{q}c_3^{q^2}+c_3^{q}+uc_3)^{q}$, which implies $u^{q}c_3^{q^2}+c_3^{q}+uc_3 \in \F_{2}$. Observe that $g(x):=u^{q}x^{q^2}+x^q+ux = xg'(x^{q-1})$ where $g'(x) = u^{q}x^{q+1}+x+u$. We show that $g(x)$ permutes $\M$, i.e., $\ker(g)=\{0\}$. This is clearly the case if $g'$ has no root in $\M$. Assume to the contrary $g'(x)=u^{q}x^{q+1}+x+u=0$. Substituting $x \mapsto x/u$ and then multiplying the entire equation by $u$ yields $x^{q+1}+x+u^2=0$. But $f(x,1)=x^{q+1}+x+u$ has no root in $\M$, so $x^{q+1}+x+u^2=0$ has also no solution in $\M$. We conclude that $g'$ has no roots in $\M$. Thus $g$ has trivial kernel, and permutes $\M$. In particular, $g(c_3)=0$ if and only if $c_3=0$ and $g(c_3)=1$ only for one value, say $c_3'$. For $c_3=0$, Eq.~\eqref{eq:carletequiv_oben} yields $c_4=1$ (recall that $x \mapsto x+x^q$ is a $2$-to-$1$ mapping since $\Gcd{k}{m}=1$ and $c_4=0$ is not possible since $M$ is invertible), and Eqs.~\eqref{eq:c1}, \eqref{eq:c2} then give $c_1=1$, $c_2=0$. This solution thus yields precisely the ``trivial'' case where $M$ is the identity matrix. For the non-zero solution $c_3'$, we know that there are either $2$ or $0$ possible values for $c_4$ by Eqs.~\eqref{eq:carletequiv} and \eqref{eq:carletequiv_oben}, again since $x \mapsto c_4 +c_4^{q}$ is a $2$-to-$1$ mapping. If we find admissible values for $c_4$, the coefficients $c_1,c_2$ are uniquely determined by Eqs.~\eqref{eq:c1}, \eqref{eq:c2}. We thus have $1$ or $3$ solutions for $a=\det(M)=1$. As mentioned above, other values of $a$ are covered by simply scaling $M$ by $\gamma$ where $\gamma^{q+1}=a$. In total, the stabilizer of $f$ under $G$ has thus size $K(2^m-1)$ where $K\in \{1,3\}$. By the orbit-stabilizer theorem, the orbit of $f$ has then size $(2^m-1)^2(2^m+1)2^m/K$. Since the orbit can only contain functions $f$ such that $f(x,1)$ has no roots in $\M$, Lemma~\ref{lem:noroots} implies that we must have $K=3$ and $G$ acts transitively on the set of all biprojective polynomials $f$ such that $f(x,1)$ has no roots in $\M$.
\end{proof}

With this result we can show that, if $m$ is even, all Carlet functions are EA-equivalent to Zhou-Pott functions, i.e., a function in the Family $\mathcal{ZP}$ listed in Table~\ref{table_biproj}.

\begin{theorem} \label{thm:carlet_1}
	Let $m$ be even. All Carlet functions on $\F$ are EA-equivalent to a Zhou-Pott function.
\end{theorem}
\begin{proof}
	Let $C=(f(x,y),g(x,y))=(xy,(1,\beta,\gamma,\delta)_{q})$ be a Carlet function. We have $C(c_1x+c_2y,c_3x+c_4y)=(f(c_1x+c_2y,c_3x+c_4y),g(c_1x+c_2y,c_3x+c_4y))$. Note that 
	\[f(c_1x+c_2y,c_3x+c_4y)=c_1c_3x^2+(c_1c_4+c_2c_3)xy+c_4y^2,\]
	in particular it can be written as $N_1(x)+N_2(y)+d_1 xy$ where $N_1,N_2$ are linear functions and $d_1 \in \M^\times$ no matter the choice of $c_1,c_2,c_3,c_4$. By Lemma~\ref{lem:transitiv}, we can choose $c_1,c_2,c_3,c_4$ with $c_1c_4+c_2c_3 \neq 0$ in a way such that $g(c_1x+c_2y,c_3x+c_4y) = d_2 (1,0,0,u)_{q}$ for an arbitrary non-cube $u$ and some $d_2 \in \M^\times$. We conclude that $L\circ C \circ M +N$ is a Zhou-Pott function, when we set $L=\diag(m_{1/d_1},m_{1/d_2})$, $M=\begin{pmatrix}
		m_{c_1} & m_{c_2} \\
		m_{c_3} & m_{c_4}
	\end{pmatrix}$ and $N = \begin{pmatrix}
		N_1 & N_2 \\
		0 & 0
	\end{pmatrix}$. 
\end{proof}

For the $m$ odd case, we will employ Theorem~\ref{thm:projequiv}. 


\begin{remark}
After proving Theorem~\ref{thm:carlet_1}, we found out that the result was 
recently published as part of the PhD thesis of Christian Kaspers~\cite{kaspers_phd}. 
The proof in~\cite{kaspers_phd} is different from ours in the sense that it relies on 
a result proved in~\cite[Theorem 2.1.]{brackentan} that characterizes (using our 
language) the orbits of the polynomials $x^{q+1}-u$ under $G$, where $u$ is a 
non-cube. In particular, it does not translate to the $m$ odd case, which is left 
as an open problem in~\cite{kaspers_phd,kasperszhouZP}. In that sense, 
Lemma~\ref{lem:transitiv} can be seen as a generalization 
of~\cite[Theorem 2.1.]{brackentan} to the $m$ odd case, solving this open problem. 
Lemma~\ref{lem:transitiv} allows us
in Section~\ref{sec:inside} to completely settle the equivalence question 
for Carlet functions also in the $m$ odd case, again closing a gap left 
in~\cite{kaspers_phd,kasperszhouZP}.
\end{remark}

\section{Walsh spectra of biprojective APN functions}\label{sec_walsh}

One of the most important properties of APN functions in even dimension is their Walsh spectrum. 
\begin{definition}
	Let $F \colon \F \rightarrow \F$ be a mapping. We define
	\[W_F(b,a) = \sum_{x \in \F}(-1)^{\trace{bF(x)+ax}} \in \Z\]
	for all $a,b \in \F$. We call the multisets
	\[\lms  W_F(b,a) \colon b \in \F^\times, a\in \F \rms \text{ and }
	  \lms |W_F(b,a)|\colon b \in \F^\times, a\in \F \rms \]
	the \emph{Walsh spectrum} and the \emph{extended Walsh spectrum} of $F$, respectively.
\end{definition} 

The extended Walsh spectrum is invariant under $\CCZ$-equivalence. Most known APN functions in even dimension $n$ have the so called \emph{classical} (or Gold-like) extended Walsh spectrum, which contains the values $0,2^{n/2},2^{(n+2)/2}$ precisely $(2^n-1)2^{n-2}$ times, $(2^n-1)2^{n+1}/3$ times and $(2^n-1)2^n/3$ times, respectively. 

We will now show that all functions in the Family $\mathcal{F}_4$ have classical Walsh spectrum. This is already known for the Taniguchi, Zhou-Pott, Carlet functions as well as the functions from $\mathcal{F}_1$ and $\mathcal{F}_2$~\cite{walsh_ZP,walshspectra,kkk}, so all known infinite families of biprojective APN functions known so far share the same Walsh spectrum. For our proof, we will employ a criterion from~\cite{kkk}. In fact, we will show that all functions in the Family $\mathcal{F}_4$ are $3$-to-$1$ functions. APN $3$-to-$1$ functions are particularly interesting since they have the smallest possible image set for APN functions, which is the reason they were studied in detail in~\cite{kkk}. There, it was also shown that all quadratic (or, more generally, all plateaued) $3$-to-$1$ APN functions have the classical Walsh spectrum, which allowed simple proofs of the Walsh spectra of (among others) the Zhou-Pott functions and the functions from Families $\mathcal{F}_1$ and $\mathcal{F}_2$. The criterion states:

\begin{theorem}[\cite{kkk}] \label{prop:walsh}
	Let $n$ be even and $F \colon \F_{2^n} \rightarrow \F_{2^n}$ be a quadratic APN function such that 
	\begin{itemize}
		\item $F(0)=0$, and
		\item Every $y \in \im(F)\setminus \{0\}$ has at least $3$ preimages. 
	\end{itemize}
	Then $F(x)=0$ if and only $x=0$ and every $y \in \im(F)\setminus \{0\}$ has precisely $3$ preimages (i.e., $F$ is $3$-to-$1$). Additionally, $F$ has classical Walsh spectrum.
\end{theorem}

Using this result, determining the Walsh spectrum of the family is reduced to a simple verification.

\begin{theorem}
	All APN functions $F \colon \F_{2^n} \rightarrow \F_{2^n}$ with $n=2m$ from the Family $\mathcal{F}_4$
	are $3$-to-$1$ and have classical Walsh spectrum.
\end{theorem}
\begin{proof}
	We check the conditions of Theorem~\ref{prop:walsh}. The first condition is clearly satisfied.
	Recall that
	\[F(x,y)=(x^{q+1}+By^{q+1},x^{qQ}y+(a/B)xy^{qQ}),\]
	so $F(x,y)=F(\omega x, \omega^2 y)$ for any $\omega \in \F_4^\times$ since $qQ=2^{k+m/2}$ where $k+m/2$ is even (recall that $k$ and $m/2$ are both odd), so $\omega^{qQ}=\omega$. Thus both conditions of Theorem~\ref{prop:walsh} hold and $F$ is $3$-to-$1$ with classical Walsh spectrum.
\end{proof}

\end{document}